\documentclass[final,11pt]{amsart}
\usepackage{amsmath,amsfonts,amssymb,amscd,verbatim,comment}
\usepackage[leqno]{amsmath}
\usepackage{fullpage}
\usepackage{xcolor}
\usepackage{array}
\usepackage{enumerate}
\usepackage{xspace}
\usepackage{bm}
\usepackage[normalem]{ulem}
\usepackage{tikz}
\definecolor{darkgreen}{RGB}{55,138,0}
\usepackage{hyperref}

\DeclareMathOperator{\Hom}{Hom}
\DeclareMathOperator{\Ext}{Ext}

\newcommand{\kk}{\Bbbk}
\newcommand\okk{\overline{\kk}}

\makeatletter
\newcommand{\leqnomode}{\tagsleft@true}
\newcommand{\reqnomode}{\tagsleft@false}
\makeatother

\numberwithin{equation}{section}
\theoremstyle{plain}
 
\newtheorem{theorem}[equation]{Theorem} 
\newtheorem{conjecture}[equation]{Conjecture}
\newtheorem{corollary}[equation]{Corollary} 
\newtheorem{lemma}[equation]{Lemma}
\newtheorem{proposition}[equation]{Proposition}

\newtheorem*{theorem*}{Theorem}

\theoremstyle{definition}
\newtheorem{definition}[equation]{Definition} 
 
\newtheorem{example}[equation]{Example}
\newtheorem{hypothesis}[equation]{Hypothesis} 
\newtheorem{question}[equation]{Question}
\newtheorem{remark}[equation]{Remark}

\mathchardef\mhyphen="2D

\newcommand\NN{\mathbb N}

\newcommand\ZZ{\mathbb Z}

\newcommand\bzero{\mathbf 0}
\newcommand\bx{\mathbf x}
\newcommand\by{\mathbf y}
\newcommand\bb{\mathbf b}

\newcommand\be{\mathbf e}
\newcommand\bp{\mathbf p}
\newcommand\bv{\mathbf v}

\newcommand\cD{\mathcal D}

\newcommand\cR{\mathcal R}
\newcommand\csr{\varrho}
\newcommand\cU{\mathcal U}

\newcommand{\tensor}{\otimes}

\DeclareMathOperator{\Aut}{Aut}
\DeclareMathOperator\End{End}

\DeclareMathOperator{\gldim}{gl dim}
\DeclareMathOperator{\GKdim}{GKdim}

\DeclareMathOperator{\Spec}{Spec}
\DeclareMathOperator{\lm}{lm}
\DeclareMathOperator{\fm}{\mathfrak m}
\DeclareMathOperator{\LND}{LND}
\DeclareMathOperator{\Id}{Id}

\DeclareMathOperator\ord{ord}

\DeclareMathOperator\rk{rk}
\DeclareMathOperator\tr{tr}
\DeclareMathOperator{\MD}{MD}

\newcommand{\inv}{^{-1}}
\newcommand\iso{\cong}
\newcommand\fsl{\mathfrak{sl}}
\newcommand\cS{\mathcal S}

\newcommand{\GWAdeg}{m}
\newcommand{\GWAord}{n}

\newcommand\grp[1]{\langle #1 \rangle}
\newcommand\restr[2]{{
  \left.\kern-\nulldelimiterspace 
  #1 
  \vphantom{\big|} 
  \right|_{#2} 
  }}

\title{Reflexive hull discriminants and applications}

\author{Kenneth Chan}
\address{(Chan) Department of Mathematics, Box 354350, 
University of Washington, Seattle, Washington 98195, USA}
\email{kenhchan@math.washington.edu, ken.h.chan@gmail.com}

\author{Jason Gaddis}
\address{(Gaddis) Department of Mathematics, Miami University, 
Oxford, Ohio 45056, USA} 
\email{gaddisj@miamioh.edu}

\author{Robert Won}
\address{(Won) Department of Mathematics, The George Washington University, Washington, DC 20052, USA}
\email{robertwon@gwu.edu}

\author{James J. Zhang}
\address{(Zhang) Department of Mathematics, Box 354350, 
University of Washington, Seattle, Washington 98195, USA}
\email{zhang@math.washington.edu}

\begin{document}

\begin{abstract}
We introduce the reflexive hull discriminant as a tool to study 
noncommutative algebras that are finitely generated, but not 
necessarily free, over their centers. As an example, we compute 
the reflexive hull discriminants for quantum generalized Weyl 
algebras and use them to determine automorphism groups and other 
properties, recovering results of Su{\'a}rez-Alvarez, Vivas, and 
others. 
\end{abstract}

\makeatletter
\@namedef{subjclassname@2020}{%
  \textup{2020} Mathematics Subject Classification}
\makeatother

\maketitle
\setcounter{section}{-1}
\section{Introduction}
\label{zzsec0}

The discriminant is a powerful tool in the study of noncommutative algebras.
It has proved useful in computing automorphism groups and solving isomorphism 
problems \cite{CPWZ1,CPWZ2,CPWZ3,CYZ2,GWY}, resolving the 
Zariski cancellation problem for different families of 
noncommutative algebras \cite{BZ2,GXW,LWZ2}, and studying 
the representations of Azumaya algebras \cite{BY,WWY2,WWY1}. 
Nguyen, Trampel, and Yakimov have also established a correspondence between 
discriminants and Poisson geometry \cite{NTY}. 

In general, the discriminant can be difficult to compute by hand, although formulas exist for 
special cases \cite{BY,CYZ2,CYZ1,GKM, NTY, NTY2}.
To apply these results one typically needs the algebra 
to be a finite rank free module over its center (although 
this is not necessary in \cite{CYZ2}). Otherwise, one can use the notion of \emph{modified discriminant ideals}, 
introduced in \cite{CPWZ2}. In this case, the discriminant can 
be defined to be the gcd of the elements in the modified 
discriminant ideal. This more general notion shares many properties with the earlier definition, 
and they coincide for an algebra that is a free 
module over its center. However, the gcd of the elements 
in the modified discriminant ideal may not exist. 
To overcome this deficiency, we introduce the \emph{reflexive hull discriminant}. 

Throughout let $\kk$ be a base field. Suppose $A$ is a finitely 
generated (=affine) prime $\kk$-algebra that is a finitely generated 
module over its center $Z$ (also called {\it module-finite over 
$Z$} for short) and let $\MD(A/Z)\subseteq Z$ denote the modified 
discriminant ideal of $A/Z$ [Definition \ref{zzdef1.2}(3)]. We 
say $A/Z$ satisfies the \emph{principal closure condition} (or 
{\it PCC} for short) if there exists a normal element $d \in A$ 
such that $\MD(A/Z) A\subseteq d A$ and $\GKdim(d A/\MD(A/Z) A) 
\leq \GKdim A - 2$ (where $\GKdim$ denotes Gelfand--Kirillov 
dimension). Let $(-)^{\vee}$ denote the $Z$-dual functor 
$\Hom_Z(-,Z)$. If $Z$ is a Cohen--Macaulay (CM) domain and $A/Z$ 
satisfies the PCC (plus some mild hypotheses), then Lemma 
\ref{zzlem1.5} shows that $(\MD(A/Z) A)^{\vee \vee} = d A$. In 
this case, $d$ is called an \emph{extended reflexive hull 
discriminant} (or {\it $\overline{\cR}$-discriminant} for short) 
of $A$ over $Z$, and is denoted $\bar{\csr}(A/Z)$. We also have 
a ``non-extended'' version of a \emph{reflexive hull discriminant} 
(or {\it $\cR$-discriminant} for short) of $A$ over $Z$ which is 
denoted by $\csr(A/Z)$ (see Section \ref{zzsec1} for more details).  

The $\overline{\cR}$-discriminant (or $\cR$-discriminant) has 
the advantage of not requiring the existence of gcds (e.g., of 
the elements in the modified discriminant ideal). Even in the 
situation that a gcd of the elements in the modified discriminant 
ideal exists, the $\overline{\cR}$-discriminant is often easier 
to compute because it can be computed locally. However, in general, 
it is not clear when an $\overline{\cR}$-discriminant exists, even 
for connected graded noetherian Calabi--Yau algebras which are 
module-finite over their centers [Question \ref{zzque4.9}]. We 
are able to show that $\overline{\cR}$-discriminants exist for 
certain skew polynomial rings over their centers.

\begin{definition}
\label{zzdef0.1}
Fix an integer $n \geq 2$. Let $\bp \in M_n(\kk^\times)$ be a 
multiplicatively antisymmetric matrix (so $p_{ii}=1$ and 
$p_{ij}=p_{ji}\inv$ for all $1 \leq i < j \leq n$). The 
\emph{skew polynomial ring} $\kk_\bp[x_1,\hdots,x_n]$ is the 
$\kk$-algebra generated by $\{x_1,\ldots,x_n\}$ and subject to relations
\[ x_j x_i= p_{ij} x_i x_j, \quad 1\leq i<j\leq n.\]
\end{definition}

The algebra $\kk_\bp[x_1,\hdots,x_n]$ is an Artin--Schelter 
regular algebra of global dimension $n$

\begin{theorem}[Theorem \ref{zzthm4.6}]
\label{zzthm0.2}
Let $A$ be a skew polynomial ring $\kk_\bp[x_1,\ldots,x_n]$ 
where each $p_{ij}$ is a root of unity. Then an 
$\overline{\cR}$-discriminant of $A$ over its center exists.
\end{theorem}

Note that the $\overline{\cR}$-discriminant shares several nice 
properties with the (modified) discriminant. In many applications, 
if the (modified) discriminant ideal is not principal, one can 
use the $\overline{\cR}$-discriminant instead. We denote the set 
of the units of a ring $A$ by $A^\times$ and for $f,g \in A$, we 
write $f =_{A^\times} g$ if $f=cg$ for some $c \in A^\times$. 
The following result shows that, like the earlier notions of the 
discriminant, the $\overline{\cR}$-discriminant can be used to 
study automorphisms and locally nilpotent derivations.

\begin{theorem}[Theorem \ref{zzthm1.11}]
\label{zzthm0.3}
Assume that $A$ is a prime algebra with center $Z$. Assume further that $Z$ is an affine CM domain and that Hypothesis \ref{zzhyp1.1} holds 
for $(A,Z)$. Suppose $d$
is an $\cR$-discriminant of $A$ over $Z$.
\begin{enumerate}
\item[(1)] 
If $g \in \Aut(A)$, then $g(d)=_{A^{\times}}d$.
\item[(2)]
Suppose that ${\text{char}}\; \kk=0$ and that $A^{\times}=
\kk^{\times}$. If $\delta$ is a locally nilpotent derivation of 
$A$, then $\delta(d)=0$.
\end{enumerate}
\end{theorem}

The following result is useful in computing the 
$\overline{\cR}$-discriminant of a tensor product. A slightly 
modified version is proved in Theorem \ref{zzthm2.5} in 
geometric language.

\begin{theorem}[Theorem \ref{zzthm1.14}]
\label{zzthm0.4}
Let $A, A'$ be prime algebras which are module-finite over their centers 
$Z, Z'$ respectively. Let $w=\rk(A/Z)$ and $w'=\rk(A'/Z')$. 
Assume that $A\otimes A'$ is prime and that Hypotheses 
\ref{zzhyp1.1}(2) and \ref{zzhyp1.10} hold for $(A,Z)$ and 
$(A',Z')$. Suppose $d$ {\rm{(}}resp. $d'${\rm{)}} is an 
$\overline{\cR}$-discriminant of $A/Z$ {\rm{(}}resp. 
$A'/Z'${\rm{)}}. Then $d^{w'}\otimes(d')^w$ is an 
$\overline{\cR}$-discriminant of $(A\otimes A')/(Z\otimes Z')$.
\end{theorem}

The notion of the $\overline{\cR}$-discriminant which we introduce here is a very natural generalization of the classical discriminant.
In \cite{CYZ2}, the authors introduced a $p$-power 
discriminant to handle the automorphism and cancellation problems 
for Veronese subrings of skew polynomial rings. Similarly we 
briefly introduce the $p$-power reflexive hull discriminant 
(respectively, extended $p$-power reflexive hull discriminant), 
denoted $\cR^{p}_{v}$-discriminant (respectively, 
$\overline{\cR}^{p}_{v}$-discriminant), see Section \ref{zzsec3}. 
Several other generalizations of the classical discriminant exist in the literature, which we illustrate in the following diagram.

\bigskip

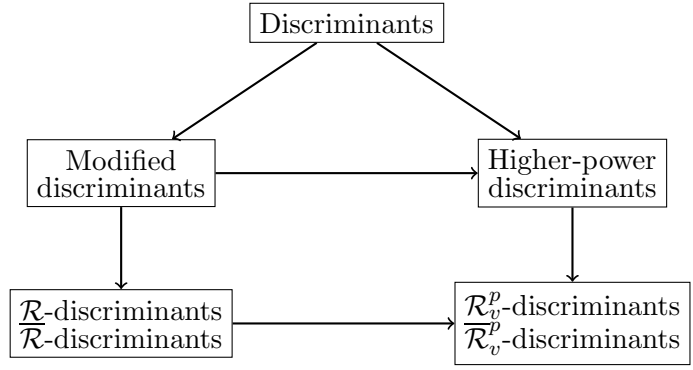
\begin{center}
\begin{tikzpicture}
\node[draw,align=center,font=\linespread{0.8}\selectfont] 
(D) at (0,2) {Discriminants};
\node[draw,align=center,font=\linespread{0.8}\selectfont] 
(MD) at (-3,0) {Modified\\ discriminants};
\node[draw,align=center,font=\linespread{0.8}\selectfont] 
(HPD) at (3,0) {Higher-power\\ discriminants};
\node[draw,align=center,font=\linespread{0.8}\selectfont] 
(RD) at (-3,-2) {$\cR$-discriminants\\ $\overline{\cR}$-discriminants};
\node[draw,align=center,font=\linespread{0.8}\selectfont] 
(RHPD) at (3,-2) {$\cR_v^{p}$-discriminants\\ $\overline{\cR}_v^p$-discriminants};
\draw[->,thick] (D) -- (MD);
\draw[->,thick] (D) -- (HPD);
\draw[->,thick] (MD) -- (RD);
\draw[->,thick] (MD) -- (HPD);
\draw[->,thick] (HPD) -- (RHPD);
\draw[->,thick] (RD) -- (RHPD);
\end{tikzpicture}
\end{center}

\medskip

To illustrate the power of the notion of the 
$\overline{\cR}$-discriminant, we study the 
$\overline{\cR}$-discriminant for a certain family of generalized 
Weyl algebras (GWAs) [Definition \ref{zzdef0.5}]. The GWA 
construction is sufficiently general as to encompass many families 
of well-known algebras, including the classical Weyl algebras and 
primitive quotients of $U(\fsl_2)$ \cite{B1}. Much of the theory 
of GWAs so far generalizes well-known properties of those families 
of algebras, see \cite{H2, H1, Jkrull, joseph, rosenberg, smith3, St2}. 
The subclass known as \emph{quantum} GWAs includes quantum planes 
and quantum Weyl algebras. A fundamental problem for any class of 
algebras is to study their automorphisms and isomorphisms within 
that class. These questions have been addressed for quantum GWAs
(of degree one) in a number of papers \cite{BJ, RS2, SAV} with the 
assumption that the base ring has Krull dimension one. In this paper 
we compute the automorphism groups for certain higher degree quantum 
GWAs and recover a number of earlier results in the degree one case.

\begin{definition}
\label{zzdef0.5}
Let $R$ be a ring, $\sigma \in \Aut(R)$, and $h$ a nonzero 
central element of $R$. The \emph{generalized Weyl algebra 
{\rm{(}}GWA{\rm{)}} of degree one} $R(x, y, \sigma, h)$ is 
generated over $R$ by $x$ and $y$ modulo the relations
\[ xy = h, \quad yx = \sigma\inv (h), \quad xr = \sigma(r)x, 
\quad yr = \sigma\inv(r) y, \]
for all $r \in R$. We say the GWA $R(x, y, \sigma, h)$ is 
\emph{quantum} if $R=\kk[t]$ and $\sigma(t)=qt$ for some 
$q \in \kk^\times$.
\end{definition}

In Section \ref{zzsec5.2}, we compute the reflexive hull 
discriminant of a degree one quantum GWA $W$ with parameter a 
root of unity (which is not 1). In this case, when $\deg_t(h)>1$, the center $Z$ of 
$W$ is a hypersurface singularity and $W$ is not free over $Z$. 
In general, both $W$ and $Z$ may have infinite global dimension.

\begin{theorem}[Theorem \ref{zzthm5.11}]
\label{zzthm0.6}
Let $W:=\kk[t](x,y,\sigma,h)$ be a quantum GWA of degree one 
with $|\sigma|=n<\infty$. Then 
$\csr(W/Z(W))=_{Z(W)^{\times}} t^{n^2(n-1)}$ and
$\bar{\csr}(W/Z(W))=_{W^{\times}} t^{n^2(n-1)}$.
\end{theorem}

Theorem~\ref{zzthm0.6} generalizes several known results. When 
$h=t$, then $W$ is a quantum plane and when $h = t-1$, then $W$ is 
isomorphic to a quantum Weyl algebra. Our result recovers the 
discriminant in both of these special cases \cite{CPWZ1, CYZ1}.

We also compute the $\cR$-($\overline{\cR}$-)discriminant for a 
class of higher degree GWAs [Theorem \ref{zzthm5.14}]. We mention 
one application here.

\begin{definition}
\label{zzdef0.7}
An algebra $A$ is \emph{cancellative} if for any algebra $B$, 
$A[x] \cong B[x]$ implies that $A \cong B$.
\end{definition}

\begin{theorem}
\label{zzthm0.8}
Let $W:=\kk[t](\bx,\by,\sigma,h)$ be a quantum GWA of degree 
$\GWAdeg> 1$ {\rm{(}}see Definition \ref{zzdef5.1}{\rm{)}}. 
\begin{enumerate}
\item[(1)]{\rm{(Theorem \ref{zzthm5.14})}}
Suppose that each $q_i$ is a root of unity with $\GWAord_i
=|\sigma_i|<\infty$, $\gcd(\GWAord_i,\GWAord_j)=1$ for all 
$i \neq j$, and $n_i>1$ for at least one $i$. Then the 
$\cR$-discriminant of $W$ over its center exists and 
$\csr(W/Z(W))=_{Z(W)^{\times}}  t^{n^2(n-1)}$,  where $n^2$ 
is the rank of $W$ over its center.
\item[(2)]{\rm{(Theorem \ref{zzthm6.6}(2)])}}
Every tensor product of finitely many GWAs as in part {\rm (1)}
and Theorem \ref{zzthm0.6} is cancellative.
\end{enumerate}
\end{theorem}

\begin{remark}
\label{zzrem0.9}
In this paper we mainly consider algebras that are module-finite 
over their centers. For algebras that are not module-finite over 
their centers, one could either
\begin{enumerate}
\item[(a)]
reduce the situation to the module-finite case by using mod $p$ 
reduction, or 
\item[(b)]
use other versions of the discriminant (e.g., the 
${\mathcal P}$-discriminants introduced in \cite{LWZ2}) and 
their reflexive hulls. 
\end{enumerate}
\end{remark}

This paper is organized as follows. We introduce the notion of 
an $\cR$-($\overline{\cR}$-)discriminant in Section \ref{zzsec1} and give a geometric interpretation in Section \ref{zzsec2}. In 
Section \ref{zzsec3} we briefly introduce the notion of an 
$\cR_v^p$-($\overline{\cR}_v^p$-)discriminant for positive 
integers $p$ and $v$. Theorem \ref{zzthm0.2} is proved in Section 
\ref{zzsec4}. In Section \ref{zzsec5} we prove Theorem 
\ref{zzthm0.6} and in Subsection \ref{zzsec5.3} we prove a 
version of Theorem \ref{zzthm0.6}, namely, Theorem \ref{zzthm0.8}(1), 
for GWAs of higher degree. We conclude, in Section \ref{zzsec6}, by demonstrating several 
applications of $\cR$-($\overline{\cR}$-)discriminants. We compute 
the automorphism group of a quantum GWA, recovering a result of 
Su\'{a}rez-Alvarez and Vivas \cite{SAV}. We also recover their result 
that the set of locally nilpotent derivations of a quantum GWA is 
trivial. We further extend these results to tensor products of 
quantum GWAs and higher degree quantum GWAs. 

\subsection*{Acknowledgments} 
R. Won was partially supported by an AMS--Simons Travel Grant and 
J. J. Zhang was partially supported by the US National Science 
Foundation (Nos. DMS-1700825 and DMS-2001015).

\section{Discriminants, modified discriminants, and 
\texorpdfstring{$\cR$-($\overline{\cR}$-)discriminants}
{(extended) reflexive hull discriminants}}
\label{zzsec1}

The goal of this section is to recall the definition of the 
(modified) discriminant and introduce a new notion, called the 
(extended) reflexive hull discriminant. Let $\kk$ be a base 
field. All algebras will be assumed to be $\kk$-algebras unless 
otherwise stated, and we write $\otimes$ for $\otimes_{\kk}$. 

Let $R$ be a commutative domain, let $B$ be an $R$-algebra, and 
assume that $R$ is a subalgebra of $B$. When $R$ is the center of 
$B$, we often use $Z$ instead of $R$. Let $F$ be a localization 
of $R$ such that $B_F := B \tensor_R F$ is finitely generated and 
free over $F$ with $w = \rk_F(B_F)<\infty$. There is a natural 
embedding of $R$-algebras given by left-multiplication:
\[ \lm: B \to B_F \to \End_F(B_F) \iso M_w(F).\]
Let $\tr_{\mathrm{int}}$ denote the usual (internal) matrix trace 
in $M_w(F)$. The \emph{regular trace} is the composition
\[ \tr_{\mathrm{reg}}:B \xrightarrow{\lm} M_w(F) 
\xrightarrow{\tr_{\mathrm{int}}} F.\]
In this paper, $\tr$ denotes the regular trace, unless otherwise
stated. We remark that this is merely for convenience, and one 
could also use the {\it standard trace} or {\it reduced trace} 
map, see \cite[Section 2.2]{BY}, instead. We often use the 
following hypothesis. 

\begin{hypothesis}
\label{zzhyp1.1} 
Let $(B,R)$ satisfy the following conditions.
\begin{enumerate}
\item[(1)]
$B$ is a prime $\kk$-algebra containing $R$ as a central 
subalgebra such that $B$ is a finitely generated $R$-module. 
\item[(2)]
The image of $\tr$ is contained in $R$. 
\end{enumerate}
\end{hypothesis}

Hypothesis \ref{zzhyp1.1}(2) is essential for several 
results in this paper. By \cite[Theorem 10.1]{Rei}, if $B$ is a 
prime ring and $R$ is the center of $B$ such that $R$ is a normal 
domain, then $(B,R)$ satisfies Hypothesis \ref{zzhyp1.1}(2). In 
the applications given in this paper, Hypothesis \ref{zzhyp1.1}(2)
can be checked easily. A general comment can be found in Lemma 
\ref{zzlem4.2}(2). 

Suppose $B$ is a prime ring. A regular normal element $x \in B$ 
\emph{divides} $y \in B$ if $y=bx$ for some $b \in B$. If $\cS$ 
is a subset of $B$, then an element $x \in B$ is a \emph{common 
divisor} of $\cS$ if (i) $x$ is a regular normal element and (ii) 
$x$ divides every $z \in \cS$. Furthermore, $x$ is the 
\emph{greatest common divisor {\rm{(}}gcd{\rm{)}}} of $\cS$ if 
any common divisor $y$ of $\cS$ divides $x$. The gcd of $\cS$, 
if it exists, is unique up to a unit in $B$.

We now recall several definitions introduced in \cite{CPWZ2}.

\begin{definition} \cite[Definition 1.2]{CPWZ2}
\label{zzdef1.2}
Assume $(B,R)$ satisfies Hypothesis \ref{zzhyp1.1}. For a positive 
integer $v$, let $\cU=\{u_i\}_{i=1}^v$ and $\cU'=\{u_i'\}_{i=1}^v$ 
be $v$-element subsets of $B$. 
\begin{enumerate}
\item[(1)] 
The \emph{discriminant} of the pair $(\cU,\cU')$ is defined to be
\[ d_v(\cU,\cU') = \det(\tr(u_iu_j')_{i,j=1}^v) \in R.\]
\item[(2)] 
The \emph{$v$-discriminant ideal} $D_v(B/R)$ is the ideal in $R$ 
generated by the set of elements $d_v(\cU,\cU)$ where $\cU$ ranges 
over all $v$-element subsets of $B$.
\item[(3)] 
The \emph{modified $v$-discriminant ideal} $\MD_v(B/R)$ is the 
ideal in $R$ generated by the set of elements $d_v(\cU,\cU')$ where 
$\cU,\cU'$ range over all $v$-element subsets of $B$. If $B$ is a 
finitely generated $R$-module of rank $w$ then we use the notation 
$\MD(B/R) := \MD_w(B/R)$.
\item[(4)] 
The \emph{$v$-discriminant} $d_v(B/R)$ is the gcd in $B$, if it 
exists, of the elements in $\MD_v(B/R)$.
\item[(5)]
The \emph{extended $v$-discriminant ideal} $\overline{\MD}_v(B/R)$ 
is the ideal in $B$ generated by $\MD_v(B/R)$, namely,
$\overline{\MD}_v(B/R)=\MD_v(B/R)B$. Similarly, we use the notation
$\overline{\MD}(B/R)$ to denote the ideal $\MD(B/R)B$.
\end{enumerate}
\end{definition}

In the special case that $B$ is free over $R$ of rank $w$, we have 
$D_w(B/R)=\MD_w(B/R)$ generated by a single element $d(B/R) := 
d_w(B/R)$, which we call the \emph{discriminant of $B$ over $R$}.

Here, we introduce a new variant of the discriminant which will be 
useful for algebras that are finitely generated, but not 
necessarily free, over their centers. 

Let $M$ be a module over a fixed commutative domain $R$. Define
$M^\vee=\Hom_R(M,R)$. The {\it reflexive hull} of $M$ is defined 
to be $M^{\vee\vee}$. There is a natural $R$-morphism
$\iota: M\to M^{\vee\vee}$ defined by 
\begin{equation}
\label{E1.2.1}\tag{E1.2.1}
\iota(x)(f)=f(x)
\end{equation}
for all $x\in M$, $f\in M^{\vee}$. It is well-known that if $I$ is 
an ideal of $R$, then $I^{\vee\vee}$ is an ideal of $R$. In fact, 
if $B$ is a prime algebra that is a finitely generated reflexive 
module over a central noetherian subalgebra and if $M$ is an ideal 
of $B$, then $M^{\vee\vee}$ is an ideal of $B$ containing $M$.

\begin{definition}
\label{zzdef1.3}
Retain the above notation. Let $(B,R)$ satisfy Hypothesis 
\ref{zzhyp1.1}. 
\begin{enumerate}
\item[(1)]
The {\it $\cR$-discriminant ideal} {\rm{(}}or {\it reflexive 
hull discriminant ideal}{\rm{)}} of $B$ over $R$ is defined 
to be
$$\cR(B/R):=(\MD(B/R))^{\vee\vee}\subseteq R.$$
\item[(2)]
If, further, $\cR(B/R)$ is a principal ideal of $R$ generated by 
an element $d$, then $d$ is called an {\it $\cR$-discriminant} 
{\rm{(}}or {\it reflexive hull discriminant}{\rm{)}} of $B$ 
over $R$ and denoted by $\csr(B/R)$. 
\item[(3)]
Suppose $B$ is a reflexive $R$-module. The 
{\it $\overline{\cR}$-discriminant ideal} {\rm{(}}or 
{\it extended reflexive hull discriminant ideal}{\rm{)}} of $B$ 
over $R$ is defined to be
$$\overline{\cR}(B/R):=
(\overline{\MD}(B/R))^{\vee\vee}\subseteq B.$$
\item[(4)]
If, further, $\overline{\cR}(B/R)$ is a principal ideal of $B$ 
generated by a normal element $d$, then $d$ is called an 
{\it $\overline{\cR}$-discriminant} {\rm{(}}or {\it extended 
reflexive hull discriminant}{\rm{)}} of $B$ over $R$ and 
denoted by $\bar{\csr}(B/R)$. 
\end{enumerate}
It is clear that $\csr(B/R)$ {\rm{(}}resp. $\bar{\csr}(B/R)${\rm{)}}, 
if it exists, is unique up to a unit in $R$ {\rm{(}}resp. $B${\rm{)}}.
\end{definition}

If $R$ is a UFD (e.g., any localization of the commutative 
polynomial ring), then every reflexive ideal of $R$ is principal and so $\csr(B/R)$ always exists. Similarly, if $B$ is a 
noncommutative noetherian UFR, then $\bar{\csr}(B/R)$ exists 
under mild conditions (see Theorem~\ref{zzthm4.8}).

Next, we introduce some conditions that are closely related to 
the existence of the $\cR$-discriminant, as well as \emph{weak} 
$\cR$-($\overline{\cR}$-)discriminants. These notions can be 
computationally useful. Throughout, we use
{\it Gelfand--Kirillov dimension}, denoted by $\GKdim$, as our dimension function.
We refer the reader to 
\cite{KL, MR} for definitions and basic properties related to 
Gelfand--Kirillov dimension.
Our use of GK dimension is not essential, and other exact dimension functions work equally well. 

\begin{definition}
\label{zzdef1.4}
Let $(B,R)$ satisfy Hypothesis \ref{zzhyp1.1}. 
Let $A$ be an algebra (which may be $B$, $R$, or another algebra). 
\begin{enumerate}
\item[(1)]
We say an ideal $I\subseteq A$ satisfies 
the {\it principal closure condition} (or 
{\it PCC}) if there exists a normal element $d \in A$ 
such that
\begin{enumerate}
\item[(a)] 
$I \subseteq d A=Ad$, and
\item[(b)] 
$\displaystyle \GKdim \left(d A/I\right) \leq 
\GKdim A - 2$.
\end{enumerate}
\item[(2)] 
We say $B/R$ satisfies the \emph{reflexive discriminant condition} 
{\rm{(}}or \emph{RDC} for short{\rm{)}} if $\MD(B/R)\subseteq R$ 
satisfies PCC for some nonzero element $d \in R$. In this case $d$ 
is called a weak $\cR$-discriminant of $B$ over $R$.
\item[(3)]
We say $B/R$ satisfies the \emph{extended reflexive discriminant 
condition} {\rm{(}}or \emph{ERDC} for short{\rm{)}} if 
$\overline{\MD}(B/R)\subseteq B$ satisfies PCC for a normal 
element $d \in B$. In this case $d$ is called a weak 
$\overline{\cR}$-discriminant of $B$ over $R$.
\end{enumerate}
The element $d$ in either part (1) or (2) or (3) {\rm{(}}if it 
exists{\rm{)}} may not be unique (even up to a unit) in general, 
unless $R$ is CM [Lemma \ref{zzlem1.5}].
\end{definition}

In many examples, RDC and ERDC are practical conditions to use. 
We will show that under some mild conditions, RDC (ERDC) 
implies the existence of a $\cR$-($\overline{\cR}$-)discriminant.

Recall from \cite[Definitions 4.1 and 4.2]{BM} that the 
(homological) \emph{grade} of a nonzero $R$-module $M$ is defined 
to be
\[j(M) = \min\{i\mid \Ext^i_R(M, R) \ne 0 \}.\] 
We say that $R$ is \emph{GK--Macaulay} if, for all nonzero finitely 
generated $R$-modules $M$, we have $\GKdim(M) + j(M) = \GKdim(R)$. 
If $R$ is an affine commutative domain, then $R$ being GK--Macaulay 
is equivalent to $R$ being CM, see 
\cite[Theorem 4.8(i)$\Leftrightarrow$(iv)]{BM}.

\begin{lemma}
\label{zzlem1.5}
Let $(B,R)$ satisfy Hypothesis \ref{zzhyp1.1}. Suppose that $R$ is 
an affine CM domain and that $B$ is a CM reflexive module over $R$.
\begin{enumerate}
\item[(1)]
Let $M$ be an ideal of $B$ satisfying PCC with respect to $d \in B$. 
Then $M^{\vee\vee}$ is a principal ideal of $B$ generated by $d$.
\item[(2)]
Let $I$ be an ideal of $R$ satisfying PCC with respect to $d \in R$. 
Then $I^{\vee\vee}$ is a principal ideal of $R$ generated by $d$.
\item[(3)]
Suppose $B/R$ satisfies RDC with respect to $d \in R$. Then 
$$\MD(B/R)^{\vee\vee}= dR\quad {\text{and}} \quad 
\csr(B/R)=_{R^{\times}}d.$$
\item[(4)]
Suppose $B/R$ satisfies ERDC with respect to $d \in B$. Then 
$$(\overline{\MD}(B/R))^{\vee\vee}= dB \quad 
{\text{and}} \quad \bar{\csr}(B/R)
=_{B^{\times}}d.$$
\end{enumerate}
\end{lemma}

\begin{proof} (1) By definition, there is a short exact sequence
\[0\to M\to d B\to d B / M \to 0.\]
Applying $(-)^{\vee}=\Hom_{R}(-,R)$ gives
\[0\to \Hom_R(dB/M, R)\to  \Hom_R(dB, R)\to  \Hom_R(M, R)\to
\Ext_R^1(dB/M,R).\]
Since $R$ is CM (and hence GK--Macaulay by \cite[p.1451]{BM} or 
\cite[Theorem 4.8(i)$\Leftrightarrow$(iv)]{BM}) and $M$ satisfies 
PCC with respect to $d$, we have $\GKdim(dB/M) \leq \GKdim B-2=
\GKdim R - 2$. Therefore, $j(dB/M) \geq 2$ by the GK--Macaulay 
property. By definition, $\Hom_R(dB/M, R)=\Ext_R^1(dB/M,R)=0$. 
Hence 
\[M^{\vee}=\Hom_R(M,R)=\Hom_R(dB,R)=(dB)^{\vee}=d^{-1}B^{\vee}.\]
Applying $(-)^{\vee}$ again yields the desired isomorphism.

(2) This is a special case of part (1) by taking $B=R$.

(3) This follows from part (2).

(4) This follows from part (1).
\end{proof}

In all applications given in this paper, the hypotheses in Lemma 
\ref{zzlem1.5} will be verified. 

A useful property of the various flavors of discriminants is their invariance under automorphisms (see 
\cite[Lemma 1.8]{CPWZ1}, \cite[Lemma 1.4]{CPWZ2}). We next show 
that the $\cR$-($\overline{\cR}$-)discriminant is preserved, up 
to a unit, by any automorphism.

\begin{lemma}
\label{zzlem1.6}
Suppose $B$ is an algebra with center $R$ and suppose that $(B,R)$ 
satisfy Hypothesis \ref{zzhyp1.1}. Assume additionally that $R$ is 
a domain. Let $G$ be a group of algebra automorphisms of $B$.
\begin{enumerate}
\item[(1)]
If $I$ is a $G$-invariant ideal of $R$, then $I^{\vee\vee}$ is also 
a $G$-invariant ideal of $R$. 
\item[(2)] 
Suppose $B$ is reflexive over $R$. Let $I$ be a $G$-invariant ideal 
of $B$. Then $I^{\vee\vee}$ is also a $G$-invariant ideal of $B$. 
\end{enumerate}
\end{lemma}

\begin{proof} 
It suffices to prove (2), since (1) follows from (2) by setting 
$B=R$. The induced action of $G$ on $I^\vee = \mathrm{Hom}_R(I,R)$ 
is given as follows: for any $g\in G$ and $\varphi\in I^\vee$, we 
have $(g\varphi)(x)=g(\varphi(g^{-1}(x)))$. Similarly 
$I^{\vee\vee}$ has an induced $G$-action.

Since $B$ is reflexive, the map $\iota:B \to B^{\vee\vee}$ as in 
\eqref{E1.2.1} is an isomorphism. Then $I^{\vee\vee}$ is the 
subset of elements $x\in B$ such that $\iota(x)(I^\vee)\subseteq R$. 
To show that $I^{\vee\vee}$ is $G$-invariant, it suffices to show 
that the induced action of $G$ on $I^{\vee\vee}$ is the same as 
the $G$-action inherited as a subset of $R$. For $x \in I$ and 
$\varphi \in I^{\vee}$, we have
\[
(g\iota(x))(\varphi) = g(\iota(x)(g^{-1}\varphi))
    = g((g^{-1}\varphi)(x))
    = gg^{-1}(\varphi(g(x)))
    = \varphi(g(x))
    = \iota(g(x))(\varphi),
\]
which proves the claim. 
\end{proof}

The $\cR$-($\overline{\cR}$-)discriminant utilizes reflexive hulls 
to obtain a good principal ideal in the case that the modified 
discriminant ideal $\MD(A/Z)$ is not principal. It may be possible 
that other operations can be employed in similar ways.

\begin{question}
\label{zzque1.7}
What are some other examples of closure operations on ideals 
(integral closure, Frobenius closure, tight closure, etc) 
which send $G$-invariant ideals to $G$-invariant ideals? 
\end{question}

In the remainder of this section, we are interested in the 
modified discriminant ideals and $\cR$-discriminants for 
tensor products of algebras. In order to study these 
$\cR$-discriminants, we recall the definition of a quasi-basis 
from \cite{CPWZ2}. We refine the definition slightly here by 
including the data of a generating set (the set $X$ below) as 
part of the data of a quasi-basis.

\begin{definition} \cite[Definition 1.10]{CPWZ2}
\label{zzdef1.8}
Let $B$ be an algebra which is module-finite over a central 
subalgebra $R$. Let $F$ be the field of fractions of $R$ (or a 
localization of $R$). Suppose that $B_F:= B \otimes_R F$ is a 
finite dimensional $F$-vector space. A set 
$\bb = \{b_1, \dots, b_w\}\subseteq B$ is a {\em{quasi-basis}} of 
$B$ with respect to a finite set $X = \{x_j\}_{j \in J}\subseteq B$ 
if the following four conditions hold:
\begin{enumerate}
\item[(a)] 
$\bb \subseteq X$,
\item[(b)] 
$\bb$ is an $F$-basis of $B_F$, where $b_i$ is viewed as 
$b_i \tensor 1 \in B_F$,
\item[(c)] 
$X$ generates $B$ as an $R$-module, and
\item[(d)] 
each $x_j$ is in the union of one-dimensional $F$-subspaces 
$\bigcup_{i=1}^w Fb_i$.  
\end{enumerate}
By part (d), for each $j \in J$, we can write $x_j=c_j b_i$ for 
some $i$ (uniquely determined by $j$) and some scalar $c_j \in F$. 
Let $C$ be the set of all such nonzero coefficients $c_j$. 
If only conditions (a)--(c) hold, we call $\bb$ a \emph{semi-basis} 
of $B$ with respect to $X$. 
\end{definition}

By Definition \ref{zzdef1.8}(d), we may define a map 
$f: J \to \{1, \dots, w\}$ where $f(j)$ is defined to be the $i$ 
such that $x_j \in F b_i$. We say $I\subseteq J$ is a 
$w$-restricted subset if $|I|=w$ and $f(I)=\{1,\ldots,w\}$. 

\begin{lemma}{\cite[Lemma 1.11(1)]{CPWZ2}}
\label{zzlem1.9} 
Let $B$ be an algebra which is module-finite over a central 
subalgebra $R$. Let $X$ be a generating set of $B$ over $R$ 
and let $\bb$ be a quasi-basis of $B$ with respect to $X$. Then, 
as ideals in $R$,
\[ \MD(B/R) = d_w(\bb,\bb) 
\left\langle c_I c_{K} \mid I, K \right\rangle.\]
where $I,K$ range over $w$-restricted subsets of $J$ and 
where $c_I=\prod_{j\in I} c_j$ and $c_K=\prod_{j\in K} c_j$.
\end{lemma}

Although many of the results above hold for more general $(B,R)$, 
for much of the remainder of the paper, we will consider the 
specific case when $R$ is the center of $B$. Henceforth, we will 
often use the notation $(A,Z)$ where $A$ is an algebra with 
center $Z$.

\begin{hypothesis}
\label{zzhyp1.10}
Let $A$ be an algebra with center $Z$. Suppose that $(A,Z)$ 
satisfy Hypothesis \ref{zzhyp1.1} and further suppose that:
\begin{enumerate}
\item[(1)]
$Z$ is affine and CM, and
\item[(2)]
there is a quasi-basis with respect to a finite generating set of 
the $Z$-module $A$.
\end{enumerate}
\end{hypothesis}

A \emph{derivation} of an algebra $A$ is a $\kk$-linear map 
$\delta:A \to A$ satisfying the Leibniz rule: 
$$\delta(ab)=\delta(a)b+a\delta(b)$$ for all $a,b \in A$. 
A derivation $\delta$ is called \emph{locally nilpotent} if 
for every $a \in A$ there exists $n \in \NN$ such that 
$\delta^n(a)=0$. 

\begin{theorem}
\label{zzthm1.11}
Suppose $(A,Z)$ satisfies Hypothesis \ref{zzhyp1.10}(1). 
\begin{enumerate}
\item[(1)] 
Suppose $d$ is an $\cR$-discriminant of $A$ over $Z$.
If $g \in \Aut(A)$, then 
$g(d)=_{Z^{\times}}d$.
\item[(2)] 
Suppose $d$ is an $\overline{\cR}$-discriminant of $A$ over $Z$.
If $g \in \Aut(A)$, then 
$g(d)=_{A^{\times}}d$.
\item[(3)]
Let $d$ be either $\csr(A/Z)$ or $\bar{\csr}(A/Z)$. Suppose 
that ${\text{char}}\; \kk=0$ and that $A^{\times}=\kk^{\times}$. 
If $\delta$ is a locally nilpotent derivation of $A$, then 
$\delta(d)=0$.
\end{enumerate}
\end{theorem}

\begin{proof}
(1) Let $g \in \Aut(A)$. Then $g$ naturally preserves the center 
$Z$ of $A$. By \cite[Lemma 1.4]{CPWZ2}, $g$ preserves the modified 
discriminant ideal $\MD(A/Z)$. Since $\MD(A/Z)^{\vee\vee}=d Z$, 
then the result follows from Lemma \ref{zzlem1.6}.

(2) The proof is similar to the proof of part (1).

(3) This result is completely analogous to 
\cite[Proposition 1.5]{CPWZ2}. 
\end{proof}

We will use the following notation for tensor products of finite 
sets. If $\mathbf{b}\subseteq A$ and $\mathbf{b'}\subseteq A'$ are 
finite sets, then we define $\mathbf{b}\otimes \mathbf{b'} = 
\{ a\otimes b \mid a\in\mathbf{b}, b\in\mathbf{b}'\}\subseteq 
A\otimes A'$.

An algebra is called {\it PI} if it satisfies a polynomial identity.

\begin{lemma} 
\label{zzlem1.12}
Let $A$ and $A'$ be prime PI algebras with centers $Z$ and $Z'$, 
respectively, such that Hypothesis \ref{zzhyp1.1} holds for both 
$(A,Z)$ and $(A',Z')$. Let $w=\rk(A/Z)$ and $w'=\rk(A'/Z')$. 
Assume that $A\otimes A'$ is prime {\rm{(}}so that the center is 
a domain{\rm{)}}. Suppose $\bb =\{b_i\}_{i=1}^{w}$ and $\bb' = \{b'_i\}_{i=1}^{w'}$ are quasi-bases for 
$A$ and $A'$ with respect to the {\rm{(}}finite{\rm{)}} generating 
sets $\mathbf{x}=\{x_j\}_{j \in J} $ and 
$\mathbf{x}'=\{x'_j\}_{j \in J'}$, respectively. Then the 
following hold. 
\begin{enumerate}
\item[(1)] 
The set $\mathbf{b}\otimes \mathbf{b}'$ is a quasi-basis of $A\otimes A'$ 
with respect to $\mathbf{x}\otimes \mathbf{x}'$. 
\item[(2)] 
$\MD(A/Z)^{w'} \otimes \MD(A'/Z')^{w}= \MD\left((A \otimes A')
/(Z \otimes Z') \right)$. As a consequence, 
\[\overline{\MD}(A/Z)^{w'} \otimes \overline{\MD}(A'/Z')^{w}= 
\overline{\MD}\left( (A \otimes A')/(Z \otimes Z') \right).\]
\item[(3)] 
If $\MD(A/Z) \subseteq dZ$ and $\MD(A'/Z') \subseteq  d' Z'$ 
for some elements $d\in Z$ and $d'\in Z'$, respectively, then
\[ \MD((A \otimes A')/(Z \otimes Z')) 
\subseteq (d^{w'} \otimes (d')^w) (Z \otimes Z').\]
\item[(4)] 
If $\overline{\MD}(A/Z) \subseteq dZ$ and $\overline{\MD}(A'/Z') 
\subseteq  d' Z'$ for some normal elements $d\in A$ and
$d'\in A'$, respectively, then
\[ \overline{\MD}((A \otimes A')/(Z \otimes Z')) 
\subseteq (d^{w'} \otimes (d')^w) (A\otimes A').\]
\end{enumerate}
\end{lemma}

\begin{proof}
It is easy to see that $Z(A\otimes A')=Z\otimes Z'$. Part (1) 
follows directly from the definition of quasi-basis, and parts 
(3,4) are consequences of part (2). It remains to show part (2).

To prove (2), we use the generators of the ideal $\MD(B/R)$ 
given by Lemma~\ref{zzlem1.9} to identify 
$\MD\left( (A \otimes A')/(Z \otimes Z') \right)$ with 
the tensor product $\MD(A/Z)^{w'} \otimes \MD(A'/Z')^{w}$. 

As above, we define a map $f: J \to \{1,\dots, w\}$ by letting $f(\alpha)$ be the unique index such that $x_{\alpha} = c_{\alpha} b_{f(\alpha)}$. We define $f': J' \to \{1, \dots, w'\}$ similarly so that $x'_{\alpha} = c'_{\alpha} b'_{f'(\alpha)}$.
Then $f\times f': J\times J'\to 
\{1,\ldots,w\}\times \{1,\ldots,w'\}$ is well-defined. 
Adopting the notation that if $I \subseteq J$ then $c_I = \prod_{i \in I} c_I$ (and similarly for subsets of $J'$ and $J \times J'$),
we can write out the generators of the modified discriminant ideals explicitly:
\begin{align*}
\MD \left(A/Z\right) &= d_w(\bb,\bb)\langle c_I c_K \mid 
{\text{$w$-restricted subsets }} I,K\subseteq J \rangle\\
\MD \left(A'/Z'\right) &= d_{w'}(\bb',\bb')\langle c'_{I'} c'_{K'} 
\mid {\text{$w'$-restricted subsets }} I',K'\subseteq J'\rangle\\
\MD\left( (A \otimes A')/(Z \otimes Z') \right) 
&= d_{ww'}(\bb\otimes\bb',\bb\otimes\bb')\\
&\qquad \langle\gamma_M \gamma_N \mid 
{\text{$ww'$-restricted subsets }} M,N\subseteq J\times J'\rangle
\end{align*}
where $\gamma_M=\prod_{(\alpha,\beta)\in M} c_\alpha \otimes 
c_\beta'$ for $M \subseteq J \times J'$. The following equality
\begin{align*}
    d_{ww'}(\bb\otimes\bb',\bb\otimes\bb')&= 
    d_{w}(\bb, \bb)^{w'}\otimes d_{w'}(\bb',\bb')^w
\end{align*}
is elementary, as it follows from properties of Kronecker 
products of matrices, so we omit its justification.

Let $\text{pr}_1$ and $\text{pr}_2$ denote the projections of 
$J\times J'$ to $J$ and $J'$, respectively. Suppose that $M,N$ are $ww'$-restricted 
subsets $J\times J'$.
For each $n \in \{1, 2, \dots, w\}$ and $n' \in \{1, 2, \dots, w'\}$ we define
\begin{align*}
I'_{n}&=(f \circ \text{pr}_1)^{-1}(n)\cap M & 
    K'_{n} &= (f\circ \text{pr}_1)^{-1}(n)\cap N \\
I_{n'}&=(f' \circ \text{pr}_2)^{-1}(n')\cap M &
    K_{n'}&=(f' \circ \text{pr}_2)^{-1}(n')\cap N.
\end{align*}


Note that since $M$ and $N$ are $ww'$-restricted subsets of $J \times J'$, we have, for each $n, n'$, that $|I_{n'}|=|K_{n'}| = w$ and $|I'_n|=|K'_n|=w'$ (and these are $w$ and $w'$-restricted subsets of $J$ and $J'$, respectively).
Then
\begin{align}
\label{E1.12.1}\tag{E1.12.1}
\gamma_M\gamma_N = \prod_{n'=1}^{w'} c_{I_{n'}}c_{K_{n'}}
\otimes \prod_{n=1}^{w}c'_{I'_n}c'_{K'_n}.
\end{align}
Hence, $d_{ww'}(\bb\otimes \bb',\bb\otimes \bb')
\gamma_M\gamma_N\in\MD(A/Z)^{w'} \otimes \MD(A'/Z')^{w}$. 
Conversely, let $C$ denote the right hand side of \eqref{E1.12.1} where $I_{n'}$, $K_{n'}$ are $w$-restricted subsets of $J$ and $I_n'$, $K_n'$ are $w'$-restricted subsets of $J'$. 
Then we can express $C$ in the form $C=\gamma_{\widetilde{M}}\gamma_{\widetilde{N}}$ where 
\begin{align*}
    \widetilde{M} &= \bigcup_{n'=1}^{w'}\bigcup_{n=1}^w (I_{n'}\times J')\cap(J\times I_n') \\
    \widetilde{N} &= \bigcup_{n'=1}^{w'}\bigcup_{n=1}^w (K_{n'}\times J')\cap(J\times K_n').
\end{align*}
Indeed, for each $1 \leq n \leq w$ and $1 \leq n' \leq w'$, we can write $(I_{n'}\times J')\cap(J\times I_n')$ (somewhat perversely) as the singleton  $\left((\restr{f}{I_{n'}})^{-1}(n), (\restr{f'}{I_{n}'})^{-1}(n')\right)$ since $I_{n'}$ being a $w$-restricted subset of $J$  means $\restr{f}{I_{n'}}$ (and similarly $\restr{f}{I_n'}$) is invertible. This shows that $\widetilde{M}$ and $\widetilde{N}$ are $ww'$-restricted subsets of $J\times J'$. It follows from the definition that $(f\circ \mathrm{pr}_1)^{-1}(n)\cap \widetilde{M} = I_n'$ and $(f'\circ \mathrm{pr}_2)^{-1}(n')\cap \widetilde{M} = I_{n'}$ (similar statements hold for $\widetilde{N}$, $K_n'$ and $K_{n'}$). 

This verifies the formula $C=\gamma_{\widetilde{M}}\gamma_{\widetilde{N}}$ which shows 
that 
\[ d_w(\bb,\bb)^{w'}d_{w'}(\bb,\bb')^w C\in \MD
\left( (A \otimes A')/(Z \otimes Z') \right).\]
This completes the proof of the first equation of (2).
The second equation is an immediate consequence of the first. 
\end{proof}

\begin{lemma}
\label{zzlem1.13}
Let $A$ be a noetherian prime PI algebra. Suppose $M_1,M_2$ are 
commuting ideals of $A$ and $d_1, d_2$ are commuting normal 
elements of $A$ such that, for $i = 1,2$,
\begin{enumerate}
\item[(a)] 
$M_i \subseteq d_i A$ and
\item[(b)] 
$\GKdim\left(d_i A/ M_i \right)  \leq \GKdim (A)-2$.
\end{enumerate}
Then 
\begin{enumerate}
\item[(1)] 
$M_1M_2 \subseteq d_1d_2 A$ and
\item[(2)] 
$\GKdim (d_1d_2A/M_1M_2) \leq \GKdim(A)-2$.
\end{enumerate}
\end{lemma}

\begin{proof}
Clearly we can assume that $d_1 d_2\neq 0$. By replacing $M_i$ 
by $d_i^{-1}M_i$ we can assume that $d_i=1$ (for $i=1,2$). We 
get a short exact sequence
\[ 0\to M_1/M_1M_2 \to A/M_1M_2 \to A/M_1 \to 0.\]
By hypothesis (b), $\GKdim(A/M_1)\le \GKdim(A) -2$ and since 
$M_1/M_1M_2$ is a finitely generated $A/M_1$-module, we also have 
$$\GKdim(M_1/M_1M_2) \le \GKdim (A/M_2)\leq \GKdim(A) -2.$$ 
Hence by additivity of $\GKdim$ on short exact sequences, 
we conclude that $\GKdim(A/M_1M_2) \le \GKdim(A) -2$. 
\end{proof}

\begin{theorem}
\label{zzthm1.14}
Retain the hypothesis of Lemma \ref{zzlem1.12}. 
\begin{enumerate}
\item[(1)]
Suppose that $A/Z$ {\rm{(}}resp. $A'/Z'${\rm{)}} satisfies 
RDC with respect to $d$ {\rm{(}}resp. $d'${\rm{)}}, a weak
$\cR$-discriminant of $A/Z$ {\rm{(}}resp. $A'/Z'${\rm{)}}. 
Then $(A\otimes A')/(Z\otimes Z')$ satisfies RDC with respect to
$d^{w'}\otimes (d')^w$, which is a weak $\cR$-discriminant of 
$(A\otimes A')/(Z\otimes Z')$.
If further Hypothesis \ref{zzhyp1.10}(1) holds for 
$(A\otimes A', Z\otimes Z')$, 
then $d^{w'}\otimes (d')^w=_{(Z\otimes Z')^\times}\csr((A\otimes A')/(Z\otimes Z'))$.
\item[(2)]
Suppose that $A/Z$ {\rm{(}}resp. $A'/Z'${\rm{)}} satisfies 
ERDC with respect to $d$ {\rm{(}}resp. $d'${\rm{)}}, a weak
$\overline{\cR}$-discriminant of $A/Z$ {\rm{(}}resp. $A'/Z'${\rm{)}}. 
Then $(A\otimes A')/(Z\otimes Z')$ satisfies ERDC with respect to $d^{w'}\otimes (d')^w$, which is a weak $\overline{\cR}$-discriminant of 
$(A\otimes A')/(Z\otimes Z')$.
If further Hypothesis \ref{zzhyp1.10}(1) holds for 
$(A\otimes A', Z\otimes Z')$, 
then $d^{w'}\otimes (d')^w=_{(A\otimes A')^\times}\bar{\csr}((A\otimes A')/(Z\otimes Z'))$.
\end{enumerate}
\end{theorem}

\begin{proof}
(1) Let $R=Z\otimes Z'$, $M_1=\MD(A/Z)^{w'}\otimes Z'$ and 
$M_2=Z\otimes \MD(A'/Z')^{w}$. Let $d_1=d^{w'}$ and $d_2=(d')^w$.
The first assertion follows from Lemmas \ref{zzlem1.12} and 
\ref{zzlem1.13} applied to $R$.

Since $Z$ and $Z'$ are affine noetherian, then so is $R$. Moreover, 
$Z$ and $Z'$ are CM so $R$ is as well by \cite[Theorem 2.1]{BK}.
The second assertion now follows from Lemma \ref{zzlem1.5}(1) and 
the first assertion. 

(2) The proof of this assertion is analogous to the proof of part (1).
\end{proof}

In \cite{GKM}, the authors studied discriminants of twisted tensor 
products of algebras.

\begin{question}
\label{zzque1.15}
Under similar hypotheses to those in \cite{GKM}, is it possible 
to compute the $\cR$-($\overline{\cR}$-) discriminant of a twisted 
tensor product of two algebras?
\end{question}

To conclude this section, we make an observation. It is easy to 
see that if $d$ is a weak $\cR$-discriminant of $B$ over $R$ 
[Definition \ref{zzdef1.4}(2)], then it is a weak 
$\overline{\cR}$-discriminant of $B$ over $R$ 
[Definition \ref{zzdef1.4}(2)]. Note that if $R$ is an affine 
CM normal domain, then the $\cR$-discriminant (resp. 
$\overline{\cR}$-discriminant) exists if and only if the weak 
$\cR$-discriminant (resp. $\overline{\cR}$-discriminant) exists 
[Lemmas \ref{zzlem1.5} and \ref{zzlem2.4}]. Combining the above 
two sentences, under some mild hypothesis, if $\csr(B/R)$ exists, 
then so does $\bar{\csr}(B/R)$ and $\bar{\csr}(B/R)=\csr(B/R)$.

\begin{lemma}
\label{zzlem1.16}
Assume Hypothesis~\ref{zzhyp1.1} for $(A,Z)$. Suppose that $Z$ is 
an affine CM normal domain and that $A$ is reflexive over $Z$. If 
$\csr(A/Z)$ exists, then so does $\bar{\csr}(A/Z)$ and 
$\bar{\csr}(A/Z)=_{A^{\times}}\csr(A/Z)$.
\end{lemma}

\section{A geometric interpretation}
\label{zzsec2}

In this section we provide a geometric motivation for the 
reflexive hull discriminant. For convenience, we assume the 
next hypothesis for a large part of this section.

\begin{hypothesis}
\label{zzhyp2.1}
Let $(A,Z)$ satisfy Hypothesis \ref{zzhyp1.1}. Let 
$X:={\text{Spec}}\; Z$ be an affine integral normal 
$\kk$-variety. Let $\cD$ (resp. $\overline{\cD}$) be the 
sheafification of the ideal $\MD(A/Z)$ in $Z$ (resp. 
$\overline{\MD}(A/Z)$ in $A$).
\end{hypothesis} 

We are interested in the reflexive hull of the modified 
discriminant ideal $\cD$ of the $\mathcal{O}_X$-order $A$ and 
consider it as a subsheaf of $\mathcal{O}_X$. Note that 
$\mathcal{D}^{\vee\vee}$ is, by definition, a reflexive sheaf of rank one, so we can construct it using the following well-known 
fact \cite[\href{https://stacks.math.columbia.edu/tag/0AY6}{Lemma 0AY6}]{stacks}.

\begin{lemma}
\label{zzlem2.2}
Let $X$ be an integral locally noetherian normal scheme. Let 
$\mathcal{L}$ be a coherent ${\mathcal O}_X$-module. The following 
are equivalent:
\begin{enumerate}
\item[(a)]
${\mathcal L}$ is reflexive,
\item[(b)]
there exists an open subscheme $\iota: U \to X$ such that
\begin{enumerate}
\item[(b1)]
every irreducible component of $X\setminus U$ has codimension 
$\geq 2$ in $X$,
\item[(b2)]
$\iota^{\ast} {\mathcal L}$ is finite locally free, and
\item[(b3)]
${\mathcal L}=\iota_{\ast} \iota^{\ast} {\mathcal L}$.
\end{enumerate}
\end{enumerate}
\end{lemma}

Let ${\mathcal L}_1$ and ${\mathcal L}_2$ be two coherent reflexive ${\mathcal O}_X$-modules. By Lemma \ref{zzlem2.2}, ${\mathcal L}_1\cong {\mathcal L}_2$ if there is an open subscheme $\iota:U\subseteq X$ such that $X \setminus U$ has codimension $\geq 2$, $\iota^\ast {\mathcal L}_1$ and $\iota^\ast {\mathcal L}_2$ are locally free, and $\iota^\ast {\mathcal L}_1 \cong \iota^\ast {\mathcal L}_2$.
The following lemma is a special case of Lemma \ref{zzlem2.2}, and is useful in 
computing reflexive hull discriminants.

\begin{lemma}
\label{zzlem2.3}
Suppose $(A,Z)$ satisfies Hypothesis \ref{zzhyp2.1}. Let $U$ be an 
open subset of $X$ such that $X \setminus U$ has codimension $\geq 2$.
\begin{enumerate}
\item[(1)] 
If there exists a normal element $d \in A$  such that the principal 
ideal $(d)$ of $A$ agrees with $\overline{\MD}(A/Z)$ on $U$, then 
$\bar{\csr}(A/Z) =_{A^\times} d$.
\item[(2)] Similarly, if there exists an element $d \in Z$  such 
that the principal ideal $(d)$ of $Z$ agrees with $\MD(A/Z)$ on 
$U$, then $\csr(A/Z) =_{Z^\times} d$.
\end{enumerate}
\end{lemma}

The following lemma is easy.

\begin{lemma}
\label{zzlem2.4}
Let $A$ be an $\mathcal{O}_X$-order and assume $(A,Z)$ satisfies 
Hypothesis \ref{zzhyp2.1}.
\begin{enumerate}
\item[(1)]
Suppose that $A$ is a CM $Z$-module. Let $\fm\in X$ be a regular 
closed point. Then 
$$\MD(A_{\fm}/Z_{\fm})=\MD(A/Z)_{\fm}=(\MD(A/Z)_{\fm})^{\vee\vee}
$$ 
is a principal, hence reflexive, ideal of $Z_{\fm}$. As a 
consequence, the support of $\cD^{\vee\vee}/\cD$ has codimension 
$\geq 2$, or equivalently, $\GKdim \left(\MD(A/Z)^{\vee\vee}/
\MD(A/Z) \right) \leq \GKdim Z-2$. 
\item[(2)]
Suppose that $A$ is a CM normal $Z$-module. Let $\fm\in X$ be a 
regular closed point. Then 
$$\overline{\MD}(A_{\fm}/Z_{\fm})=\overline{\MD}(A/Z)_{\fm}=
(\overline{\MD}(A/Z)_{\fm})^{\vee\vee} $$ 
is a principal, hence reflexive, ideal of $A_{\fm}$. As a 
consequence, the support of 
$\overline{\cD}^{\vee\vee}/\overline{\cD}$ has codimension 
$\geq 2$, or equivalently, $\GKdim\left( \overline{\MD}(A/Z)^{\vee\vee}/
\overline{\MD}(A/Z)\right)\leq \GKdim A-2$. 
\item[(3)]
Suppose that $V$ is an affine open subset of $X$ such that 
$\restr{A}{V}$ is locally free. Then 
$$\restr{\MD(A/Z)}{V} = \MD(\restr{A}{V}/{\mathcal O}(V)) 
= \mathcal{O}_V(D)$$ 
where $D$ is given by the zero locus of the usual 
discriminant $d\in\mathcal{O}_V(V)$.
\end{enumerate} 
\end{lemma}

\begin{proof}
(1) Since $\fm$ is a smooth closed point, $Z_{\fm}$ is local and 
regular. Since $A$ is CM over $Z$, $A_{\fm}$ is CM over $Z_{\fm}$. 
By the Auslander--Buchsbaum formula, $A_{\fm}$ is a finite 
projective (hence free) module over $Z_{\fm}$. So the (modified) 
discriminant ideal is principal and hence reflexive. 

The consequence is clear.

(2) The proof is similar to the proof of part (1).

(3) By commutative algebra, taking reflexive hulls commutes 
with localization 
as does the formation of modified discriminants. Thus the 
assertion follows. 
\end{proof}

Let $U$ be the maximal open subvariety such that 
$\restr{\mathcal{D}}{U}$ is locally free. When $A$ is CM over $Z$, 
then $U$ contains the non-singular locus of $X$ by Lemma 
\ref{zzlem2.4}(1) (so the $\cR$-discriminant exists over an open 
subvariety whose complement has codimension $\geq 2$). We can write 
$\restr{\cD}{U}=\restr{\mathcal{D}^{\vee\vee}}{U}=\mathcal{O}_U(D)$ 
where $D$ is a Cartier divisor on $U$. Denote by $\iota:U\to X$ the 
inclusion map, then 
$\mathcal{D}^{\vee\vee} = \iota_\ast \mathcal{O}_U(D)$ by Lemma 
\ref{zzlem2.2}. Therefore one might consider this Cartier divisor 
as a ``shadow'' of the reflexive hull discriminant 
${\mathcal D}^{\vee\vee}$. In other words, the $\cR$-discriminant 
is a ``completion'' or ``closure'' of this Cartier divisor. A 
similar comment can be made for $\overline{\cD}$ which is 
considered as a subsheaf of (the sheafification of) $A$. It is 
interesting to work out the closed subvarieties so that $\cD$ 
(resp. $\cD^{\vee\vee}$, $\overline{\cD}$, and 
$\overline{\cD}^{\vee\vee}$) are not locally free.

The upshot of Lemmas \ref{zzlem2.3} and \ref{zzlem2.4} is that 
to compute $\mathcal{D}^{\vee\vee}$ (resp. 
$\overline{\cD}^{\vee\vee}$) we do not have to compute the 
modified discriminant ideal, i.e., we do not have to consider 
all $\mathrm{rk}(A/Z)\times \mathrm{rk}(A/Z)$-minors of a large 
matrix. This reduces the computation significantly. Once we have 
determined $D$ as a Cartier divisor on $U$ (or equivalently, a 
locally free ideal of $\mathcal{O}(U)$), the question of whether 
the reflexive hull of the modified discriminant ideal is 
principal becomes a question in algebraic geometry: is 
$\iota_\ast\mathcal{O}_U(D)$ locally free? Or equivalently, does 
the Cartier divisor $D$ on $U$ extend to a Cartier divisor on all 
of $X$? 
We illustrate this recipe by computing the $\cR$-discriminant of 
several classes of algebras in the following examples as well as in
Sections \ref{zzsec4} and \ref{zzsec5}.

\begin{example}
\label{zzex2.6}
Suppose ${\rm{char}}\; \kk\neq 2$. Let $A=\kk_\bp[x_1,x_2,x_3]$ 
as in Definition \ref{zzdef0.1} (for $n=3$). Note that $\csr(A/Z)$ 
exists by Theorem \ref{zzthm4.6}.
\begin{enumerate}
\item[(1)]
Let $(p_{12},p_{13},p_{23})=(-1,1,1)$. The center of $A$ is 
$Z=\kk[u, v, w]$ where $u = x_1^2, v = x_2^2$, and $w = x_3$. The rank of 
$A$ over $Z$ is 4. The discriminant, modified discriminant, and 
reflexive hull discriminant are all equal to $\csr(A/Z)=
\bar{\csr}(A/Z)=_{\kk^\times}(uv)^2=x_1^4x_2^4$.
\item[(2)]
Let $(p_{12},p_{13},p_{23})=(-1,-1,1)$. The center of $A$ is 
$\kk[u,v,w,z]/(vw-z^2)$ where $u=x_1^2$, $v=x_2^2$, $w=x_3^2$, 
and $z=x_2x_3$. The rank of $A$ over $Z$ is 4. Let $V$ be the 
open subset of $X$ with $v\neq 0$. Over $V$, the discriminant of 
$A_{v}$ is $(uv)^2$, or simply $u^2$, since $v^2$ is a unit in 
$Z_v$. Let $W$ be the open subset of $X$ with $w\neq 0$. Similarly, 
over $W$, the discriminant of $A_{w}$ is $u^2$. Note that 
$X\setminus (V\cup W)$ has codimension 2, so the $\cR$-discriminant 
is $\csr(A/Z)=\bar{\csr}(A/Z)=_{\kk^\times}u^2=x_1^4$, which agrees with the 
modified discriminant given in \cite[Example 1.3(3)]{CPWZ2}. The 
modified discriminant ideal is generated by 
$\{x_1^4 x_2^i x_3^{4-i}\}_{i=0}^4$ by \cite[Example 1.3(3)]{CPWZ2}.
\item[(3)]
Let $(p_{12},p_{13},p_{23})= (-1,-1,-1)$. The center of $A$ is 
$\kk[u, v, w, z]/(uvw - z^2)$ where $u = x_1^2$, $v = x_2^2$, $w = x_3^2$, 
and $z = x_1x_2x_3$. Let $U_3$ be the open subset of $X$ with 
$uv \neq 0$. Over $U_3$, the discriminant of $A_{uv}$ is $(x_1x_2)^4$ 
which is equivalent to $1$. Let $U_2$ (resp. $U_1$) be the open 
subset of $X$ with $uw \neq 0$ (resp, $vw \neq 0$). Since 
$X\setminus(U_1\cup U_2\cup U_3)$ has codimension 2 in $X$, the 
$\cR$-discriminant of $A/Z$ is $\csr(A/Z)=\bar{\csr}(A/Z)=_{\kk^\times} 1$.
\end{enumerate}
\end{example}

To conclude this section we prove a geometric version of Theorem 
\ref{zzthm1.14}.

\begin{theorem}
\label{zzthm2.5}
Let $\kk$ be an algebraically closed field. Suppose that
\begin{enumerate}
\item[(a)]
$A$ is a prime algebra that is module-finite over its center $Z$,
\item[(b)]
$Z$ is an affine normal CM domain and that $\csr(A/Z)$ exists, and
\item[(c)]
$A$ is CM as a module over $Z$.
\end{enumerate}
Let $(A',Z')$ be another pair satisfying {\rm{(a)--(c)}}. Let $w$ 
{\rm{(}}resp. $w'${\rm{)}} be the rank of $A$ {\rm{(}}resp. 
$A'${\rm{)}} over $Z$ {\rm{(}}resp. $Z'${\rm{)}}. Further assume 
that $A\otimes A'$ is prime. Then the following hold.
\begin{enumerate}
\item[(1)]
$Z\otimes Z'$ is a normal CM domain.
\item[(2)]
$A\otimes A'$ is a prime ring with center $Z\otimes Z'$ and 
$A \otimes A'$ is CM as a module over $Z\otimes Z'$.
\item[(3)]
$\csr(A\otimes A'/Z\otimes Z') =_{(A\otimes A')^\times} \csr(A/Z)^{w'} \otimes \csr(A/Z')^{w}$.
\item[(3')]
Suppose $A\otimes A'$ is reflexive over $Z\otimes Z'$. If 
$\csr(A/Z)$ is replaced by $\bar{\csr}(A/Z)$ in {\rm{(b)}}, then 
$\bar{\csr}(A\otimes A'/Z\otimes Z')=_{(Z\otimes Z')^\times}\bar{\csr}(A/Z)^{w'} 
\otimes \bar{\csr}(A/Z')^{w}$.
\end{enumerate}
\end{theorem}

\begin{proof}
(1) By \cite[Lemma 1.1]{RoS}, $Z\otimes Z'$ is a domain. That 
$Z \tensor Z'$ is CM follows by the same argument as in Theorem 
\ref{zzthm1.14}. Since $\kk$ is algebraically closed,
every simple module over $Z\otimes Z'$ is a tensor 
product of simple modules over $Z$ and $Z'$ respectively. 
Thus, singular points on $\Spec Z\otimes Z'$ are of the
form $(z, z')$ where either $z\in \Spec Z$ or $z'\in \Spec Z'$ 
is singular. This implies that the singular locus 
of $\Spec Z\otimes Z'$ has codimension $\geq 2$. 
By Serre's criterion for normality, $Z\otimes Z'$ is normal.

(2) It is clear that the center of $A\otimes A'$ is $Z\otimes Z'$. 
Since $Z$ is affine and CM, by Noether normalization, there is a 
polynomial subring $S\subseteq Z$ such that $Z$ is a finite free 
module over $S$. Since $A$ is CM over $Z$, $A$ is a finite free 
module over $S$. Similarly, there is a polynomial subring $S' 
\subseteq Z'$ such that both $Z'$ and $A'$ are finite free module 
over $S'$. Then $A\otimes A'$ is finite and free over the central
polynomial subring $S\otimes S'$. Hence $A\otimes A'$ is CM over 
$S\otimes S'$ and over $Z\otimes Z'$. 

(3) Let $U$ be the non-singular locus of $X:=\Spec Z$. Then 
$X\setminus U$ has codimension $\geq 2$ in $X$ as $Z$ is normal. 
Similarly, the non-singular locus of $X':=\Spec Z'$, denoted by 
$U'$, has complement with codimension $\geq 2$ in $X'$. Since 
$\kk$ is algebraically closed, $U\times U'$ is an open scheme of 
the non-singular locus of $X\times X':=\Spec (Z\otimes Z')$ whose 
complement has codimension $\geq 2$ in $X\times X'$. By Lemma 
\ref{zzlem2.4}(1) and a local version of Lemma \ref{zzlem1.12} 
(with $\bb=\mathbf{x}$ and $\bb'=\mathbf{x}'$), over $U\times U'$, 
$\MD(A\otimes A'/Z\otimes Z')$ is equal to 
$\MD(A\otimes A'/Z\otimes Z')^{\vee\vee}$ and equal to the Cartier 
divisor ${\mathcal O}_{U\times U'}(D)$ where $D$ is determined by 
$d := \csr(A/Z)^{w'} \csr(A/Z')^{w}$. Since $d$ is defined over 
$X\times X'$, $\MD(A\otimes A'/Z\otimes Z')^{\vee\vee}$ is the 
principal ideal generated by $d$ by Lemma \ref{zzlem2.2}. The 
assertion follows.

(3') This is similar to the proof of (3).

\end{proof}

\section{\texorpdfstring{$\cR^{p}_{v}$-discriminants}
{Higher reflexive hull discriminants}}
\label{zzsec3}


Next, we introduce the $p$-power reflexive hull $v$-discriminant (or $\cR^{p}_v$-discriminant), as well as its extended counterpart, the $\overline{\cR}^{p}_v$-discriminant.
These notions generalize the $p$-power discriminants introduced in \cite{CYZ2}.
As we illustrate in Example~\ref{zzex3.2}, there are situations in which the $\cR$-discriminant does not exist, 
but for some $p$, the $\cR^{p}_{v}$-discriminants do exist. 
Hence, in these situations, the $\cR^{p}_{v}$-discriminant can serve as a useful invariant (see Theorem~\ref{zzthm3.3}).
For the sake of brevity, in this section, we omit some non-essential details.

\begin{definition} 
\label{zzdef3.1}
Assume Hypothesis \ref{zzhyp1.1} for $(B,R)$. Fix two positive 
integers $p,v$.
\begin{enumerate}
\item[(1)] 
The \emph{$p$-power $v$-discriminant ideal}, denoted by 
$\MD^{p}_v(B/R)$, is the ideal $(\MD_v(B/R))^p$ (where $\MD_v(B/R)$ 
is given in Definition \ref{zzdef1.2}(3)).
\item[(2)]
The {\it $p$-power reflexive hull $v$-discriminant ideal} or 
{\it $\cR^{p}_v$-discriminant ideal} of $B$ over $R$ is defined 
to be
$$\cR^{p}_v(B/R)=(\MD^{p}_v(B/R))^{\vee\vee}.$$
\item[(3)]
If further $\cR^{p}_v(B/R)$ is a principal ideal generated by an 
element $d$ in $R$, then $d$ is called the {\it $p$-power 
reflexive hull $v$-discriminant} or {\it $\cR^{p}_v$-discriminant} 
of $B$ over $R$ and denoted by $\csr^{[p]}_v(B/R)$. 
\item[(4)]
The {\it $p$-power extended reflexive hull $v$-discriminant ideal} 
or {\it $\overline{\cR}^{p}_v$-discriminant ideal} of $B$ over $R$ 
is defined to be
$$\overline{\cR}^{p}_v(B/R)
=\left(\overline{\MD}^{p}_v(B/R)\right)^{\vee\vee}.$$
\item[(5)]
If further $\overline{\cR}^{p}_v(B/R)$ is a principal ideal 
generated by a normal element $d$ in $B$, then $d$ is called the 
{\it $p$-power extended reflexive hull $v$-discriminant} or 
{\it $\overline{\cR}^{p}_v$-discriminant} of $B$ over $R$ and 
denoted by $\bar{\csr}^{[p]}_v(B/R)$. 
\end{enumerate}
It is clear that if it exists, $\csr^{[p]}_v(B/R)$ {\rm{(}}resp. 
$\bar{\csr}^{[p]}_v(B/R)${\rm{)}} is unique up to 
a unit in $R$ {\rm{(}}resp. in $B${\rm{)}}.
\end{definition}

Theorem \ref{zzthm4.8} in the next section indicates that if 
$R$ (resp. $B$) is nice enough, then $\csr^{[p]}_v(B/R)$ (resp. 
$\bar{\csr}^{[p]}_v(B/R)$) exists. In general, 
the existence of $\csr^{[p]}_v(B/R)$ (resp. 
$\bar{\csr}^{[p]}_v(B/R)$) is dependent on $(p,v)$ as the 
next example shows.

\begin{example}
\label{zzex3.2}
Let $Z = \kk[a,b,c]/(ab-c^3)$ and consider the $A_2$ singularity $X=\mathrm{Spec}(Z)$. The divisor class group $\mathrm{Div}(X)$ 
of $X$ is isomorphic to $\ZZ/3\ZZ$ with generator given by $I=(a,c)$. Let $A$ 
be the matrix algebra $\begin{pmatrix} Z & I\\ Z& Z\end{pmatrix}$. 
Then the modified discriminant is $\MD(A/Z)=I^2$. The rank of $A$ 
over its center $Z$ is 4. 
\begin{enumerate}
\item[(1)]
For every $q \geq 1$, $I^q$ is not principal, and so no $p$-power discriminant ideal $\MD(A/Z)^p$ is  principal.
\item[(2)]
The $\cR$-discriminant ideal $(I^2)^{\vee\vee}$ is not 
principal as $\mathrm{Div}(X)\cong \ZZ/3\ZZ$ with generator 
given by $I=(a,c)$. So $\csr(A/Z)$ (namely, $\csr^{[1]}_4(A/Z)$)
does not exist. 
\item[(3)]
Since $I^3=(a^3,a^2c,ac^2,c^3)=a(a^2,ac,c^2,b)$ and 
$\GKdim(Z/(a^2,ac,c^2,b))=0$ we have, by Lemma \ref{zzlem1.5}(1), 
that $(I^3)^{\vee\vee}=(a)$ is principal. Hence 
$(I^6)^{\vee\vee}=(a^2)$ is also principal. This means that 
$\csr^{[3]}_4(A/Z)=_{\kk^\times}a^2$.
\item[(4)]
One can check that 
$$\csr^{[p]}_w(A/Z)=\begin{cases}
1 & 1\leq w \leq 3, \quad p\geq 1,\\
a^{\frac{p}{3}} & w=3, \quad 3\mid p,\\
{\text{does not exist}} & w=3 {\text{ or }} w=4, \quad 3\nmid p,\\
a^{\frac{2p}{3}} & w=4, \quad 3\mid p, \\
0 & w\geq 5, \quad p\geq 1.
\end{cases}$$
\end{enumerate}
\end{example}

The proof of the following result is similar to the proof of 
Theorem \ref{zzthm1.11} and is omitted. 

\begin{theorem}
\label{zzthm3.3}
Let $(A,Z)$ satisfy Hypotheses \ref{zzhyp1.1} and \ref{zzhyp1.10}(1) 
where $Z$ is the center of $A$. Fix two positive integers $p,v$.
\begin{enumerate}
\item[(1)] 
Suppose $d$ is an $\cR^{p}_v$-discriminant of $A$ over $Z$.
If $g \in \Aut(A)$, then 
$g(d)=_{Z^{\times}}d$.
\item[(2)] 
Suppose $d$ is an $\overline{\cR}^{p}_v$-discriminant of $A$ over $Z$.
If $g \in \Aut(A)$, then 
$g(d)=_{A^{\times}}d$.
\item[(3)]
Let $d$ be either $\csr(A/Z)$ or $\bar{\csr}(A/Z)$. Suppose 
that ${\text{char}}\; \kk=0$ and that $A^{\times}=\kk^{\times}$. 
If $\delta$ is a locally nilpotent derivation of $A$, then 
$\delta(d)=0$.
\end{enumerate}
\end{theorem}

\section{Proof of Theorem \ref{zzthm0.2}}
\label{zzsec4}

In this section we first make a few elementary remarks about 
Hypotheses \ref{zzhyp1.1}, \ref{zzhyp1.10}, and \ref{zzhyp2.1}
and then prove Theorem \ref{zzthm0.2}.

Let $A$ be a prime affine algebra that is module-finite over 
its center $Z:=Z(A)$. Following Brown and Hajarnavis \cite{BH}, 
$A$ is called {\it homologically homogeneous} (or {\it hom-hom} 
for short) of dimension $d$ if all simple $A$-modules have the 
same projective dimension $d$ (see also \cite{SVdB, SZ}). 
Hom-hom rings appear naturally in several contexts.

\begin{example}
\label{zzex4.1} 
The following are examples of hom-hom rings:
\begin{enumerate}
\item[(1)]
Affine noetherian prime Hopf algebras that are module-finite over 
their centers and have finite global dimension \cite[Theorem A]{BG}.
\item[(2)]
Connected graded noetherian Artin--Schelter regular PI algebras
\cite{SZ}. These include the PI skew polynomial rings defined in 
Definition \ref{zzdef0.1}.
\item[(3)]
Noncommutative crepant resolutions, in the sense of Van den Bergh
\cite{VdB2}. 
\end{enumerate}
\end{example}

The results in the following lemma are well-known.

\begin{lemma} \cite{BH, SVdB, SZ}
\label{zzlem4.2} 
Let $A$ be a hom-hom ring with center $Z$. 
\begin{enumerate}
\item[(1)]
$Z$ is normal. As a consequence, Hypothesis \ref{zzhyp2.1} holds.
\item[(2)]
The image of $\tr$ is in $Z$. As a consequence, Hypothesis 
\ref{zzhyp1.1} holds.
\item[(3)]
Suppose $Z$ is affine. Then $A$ is a reflexive module over $Z$. 
\end{enumerate}
\end{lemma}

\begin{proof} 
(1) This is \cite[Theorem 6.1]{BH} (or \cite[Theorem 5.6(ii)]{SZ}).
Note that a Krull domain is just a noetherian normal domain.

(2) By \cite[Theorem 5.4(iii)]{SZ}, $A$ is equal to its trace 
ring, which implies, by definition, that the image of $\tr$ is 
in $Z$. 

(3) By \cite[Theorem 2.3(1,4)]{SVdB}, $A$ is a CM tame $Z$-order 
in the sense of \cite[Section 2]{SVdB}. So $A$ is a finitely 
generated CM (and then free) module over a polynomial subring $R$ 
of $Z$. Hence $A$ is reflexive over $R$. It follows from 
\cite[Lemma 2.1]{SVdB} that $A$ is also reflexive over $Z$. 
\end{proof}

\begin{lemma}
\label{zzlem4.3} 
Let $A$ be a prime ring with center $Z$. Suppose that $Z$ is 
normal and that ${\text{char}}\; \kk$ does not divide the rank 
of $A$ over $Z$. Then $Z$ is CM. As a consequence, Hypothesis 
\ref{zzhyp1.10}(1) holds.
\end{lemma}

\begin{proof} If ${\text{char}}\; \kk=0$, then the result follows from 
\cite[Theorem 2.3(3)]{SVdB}. Now suppose that 
${\text{char}}\; \kk>0$.
Since ${\text{char}}\; \kk$ does not divide the rank of $A$ over 
$Z$, $\tr$ is a nonzero scalar multiple of the identity map 
when restricted to $Z$. Then $Z$ is a $Z$-module direct summand 
of $A$, and the assertion follows from \cite[Theorem 2.5]{BH}.
\end{proof}

Note that Hypothesis \ref{zzhyp1.10}(1) fails if ${\text{char}}\; 
\kk$ divides ${\rm{rk}}_Z(A_Z)$, see \cite[Example 7.3]{BHM} for 
a local example. An affine version can be made using the idea
of \cite[Example 7.3]{BHM}, see explanation in 
\cite[Example 2.3(vi)]{BM}.

Hypothesis \ref{zzhyp1.10}(2) is a technical condition and we 
expect that Lemma \ref{zzlem1.12} and Theorem \ref{zzthm1.14}
hold without this hypothesis, see Theorem \ref{zzthm2.5}. 

For a hom-hom ring $A$, the dualizing module $\omega_A$ is 
defined in \cite[p.663]{SVdB}. We say $A$ is {\it Calabi--Yau}
if $\omega_A\cong A$ as $A$-bimodule \cite[Remark 3.2]{SVdB}.
When $A$ is a prime affine algebra and module-finite over its 
center, this definition is equivalent to other definitions, 
for example, the one given given by Ginzburg \cite{ginz}.

\begin{lemma}
\label{zzlem4.4}
Let $A$ be an affine Calabi--Yau prime algebra that is 
module-finite over its center $Z$. Then $Z$ is Gorenstein.
\end{lemma}

\begin{proof}
This follows from \cite[Lemma 2.5(5)]{SVdB}.
\end{proof}

For the rest of this section we prove Theorem \ref{zzthm0.2}.
The skew polynomial ring $\kk_\bp[x_1,\hdots,x_n]$, as in 
Definition \ref{zzdef0.1}, is $\ZZ^n$-graded with $\deg x_i=e_i$. 
It is also connected $\NN$-graded when we set $\deg x_i=1$ for 
all $i$.

\begin{lemma}
\label{zzlem4.5}
Let $A$ be the skew polynomial ring $\kk_\bp[x_1,\ldots,x_n]$ 
such that each $p_{ij}$ is a root of unity. Let $r^2$ be the 
rank of $A$ over its center $Z$. 
\begin{enumerate}
\item[(1)]
${\rm{char}}\; \kk$ and $r$ are coprime. As a consequence, $Z$ 
is a CM normal domain.
\item[(2)]
Let $M$ be an ideal of $A$ generated by monomials, namely,
${\mathbb Z}^n$-homogeneous elements. Then $M^{\vee\vee}$ 
is a principal ideal of $A$.  
\item[(3)]
For every monomial $f$, $\tr(f)=\begin{cases} r^2 f & f\in Z\\
0 & f\not\in Z.\end{cases}$
\item[(4)]
Let $\cU$ and $\cU'$ be two $v$-element sets of monomials. Then  
\[ d_v(\cU, \cU')\in Z\cap \kk\prod_{f\in \cU, f'\in \cU'} ff',\]
which is a monomial in $Z$.
\item[(5)]
For every pair of positive integers $(p,v)$, $\MD_v^p(A/Z)$ is generated by a set of 
monomials in $Z$.
\item[(6)]
For every pair of positive integers $(p,v)$, $\overline{\MD}_v^p(A/Z)$ is an ideal of 
$A$ generated by a set of monomials in $Z$.
\end{enumerate}
\end{lemma}

\begin{proof}
(1) For each $i,j$, let $n_{ij}$ denote the order of $p_{ij}$. If ${\rm{char}}\; \kk > 0$,
then ${\rm{char}}\;\kk$ is coprime to $n_{ij}$ for all $i,j$. Since the rank of 
$A$ over $Z$ is a factor of $(\prod_{i<j} n_{ij})^2$, it follows that $r$ and ${\rm{char}}\; 
\kk$ are coprime. By Lemma \ref{zzlem4.3}, $Z$ is CM and by Lemma 
\ref{zzlem4.2}(1), $Z$ is normal.

(2) Let $M$ be an ideal of $A$ generated by a set $\cS$ of monomials.
We can write $M=\cS A$ since each element in $\cS$ is normal in $A$. 
For each $i$, there is a positive integer $w_i$ such that 
$a_i:= x_i^{w_i}$ is in $Z$. Let $U_i$ denote the open 
subset of $X:=\Spec Z$ with $\prod_{j\neq i} a_j\neq 0$.  

Let $U:=\bigcup_{i=1}^n U_i$. First, we claim that $X \setminus U$ has 
codimension $\geq 2$. Note that the subalgebra $R=\kk[a_1,\ldots,a_n]$ 
is a polynomial subring of $Z$ such that $A$ is finitely generated 
over $R$. Since $A$ and $Z$ are CM, both $A$ and $Z$ are finitely 
generated free $R$-modules. As a consequence, $\{a_1,\ldots,a_n\}$ 
is an $R$-regular sequence (and also a $Z$-regular sequence) in $A$ 
and in $Z$. Write $Z$ as a factor ring 
$\kk[a_1,\ldots,a_n,b_1,\ldots,b_w]/I$ for some generators 
$b_1,\ldots, b_w$. If $\fm$ is a closed point which is not in $U$, then 
there are at least two $i_0\neq j_0$ such that $a_{i_0}=0=a_{j_0}$. 
Since  $\{a_1,\ldots,a_n\}$ is a $Z$-regular sequence, 
$\GKdim(Z/(a_{i_0},a_{j_0})) = \GKdim Z-2=n-2$. 
Therefore $X\setminus U$ has dimension at most $n-2$ and we have proved the claim. 

Now for each $f\in \cS$, write $f=x_1^{h_1}\cdots x_n^{h_n}$ and 
define $p_i(f):=h_i$ for each $i$. Let $s_i = \min\{p_i(f) \mid f \in \cS \}$ and define $d = x_1^{s_1}\cdots x_n^{s_n}$. 
It is clear that $d$ is the gcd of the elements in $\cS$ and 
that $M\subseteq dA$. 
We now claim that $M^{\vee\vee}$ is 
the principal ideal of $A$ generated by $d$. Since $A$ is 
$Z$-reflexive [Lemma \ref{zzlem4.2}(3)], $dA$ is reflexive and
contains $M^{\vee\vee}$. By Lemma \ref{zzlem2.2}, it suffices 
to show that $\restr{(M^{\vee\vee})}{U} =\restr{(dA)}{U}$, or 
equivalently, that for each $i$, we have $\restr{(M^{\vee\vee})}{U_i}
=\restr{(dA)}{U_i}$. Over $U_i$, we invert the elements 
$a_j=x^{w_j}$ for all $j\neq i$. Note that 
\[\restr{A}{U_i}=A[a_1^{-1},\ldots,
\widehat{a_{i}^{-1}}, \ldots, a_n^{-1}]=
A[x_1^{-1},\ldots,
\widehat{x_{i}^{-1}}, \ldots, x_n^{-1}]
=\kk_{\bp}[x_{1}^{\pm 1},\ldots,x_{i-1}^{\pm 1},
x_{i},x_{i+1}^{\pm 1},\ldots,x_{n}^{\pm 1}].
\]
Hence
$$\restr{M}{U_i}
=\restr{(\cS A)}{U_i}
=\sum_{f\in \cS} f \restr{A}{U_i}
=\sum_{f\in \cS} x_i^{p_i(f)} \restr{A}{U_i}
=x_i^{s_i} \restr{A}{U_i}
=\restr{(d A)}{U_i}$$
which implies that $\restr{(M^{\vee\vee})}{U_i}=
\restr{(d A)}{U_i}$ as required. This proves the claim. 

(3) Consider $A$ as a $\ZZ^n$-graded algebra. Choose a semi-basis 
of $A$ over $Z$ to consist of $\ZZ^n$-homogeneous elements (namely, 
monomials). By linear algebra, for each $\ZZ^n$-homogeneous 
element $f$, $\tr(f)$ is either 0 or equal to $r^2 f$. The second 
case happens if and only if $f$ is in $Z$. 

(4) This is \cite[Lemma 2.6]{CPWZ2}.

(5, 6) We can choose a generating set of $A$ over $Z$ consisting 
of $\ZZ^n$-homogeneous elements. Both assertions follow from part 
(4).
\end{proof}

\begin{theorem}
\label{zzthm4.6}
Let $A$ be a skew polynomial ring $\kk_\bp[x_1,\ldots,x_n]$ 
where each $p_{ij}$ is a root of unity. Then, for any two 
positive integers $p,v$, $\bar{\csr}^{[p]}_v(A/Z)$ exists.
\end{theorem}

\begin{proof} 
By Lemma \ref{zzlem4.5}(6), $\overline{\MD}^p_v(A/Z)$ is 
generated by a set of monomials. Now the assertion follows from 
Lemma \ref{zzlem4.5}(2).
\end{proof}

By the proof of Theorem \ref{zzthm4.6}, one can verify that in 
this case, $\bar{\csr}(A/Z)$ agrees with the usual discriminant 
$d_{r^2}(A/Z)$ defined in \cite[Definition 1.2(3)]{CPWZ2}.

In Example \ref{zzex2.6}, we computed the $\mathcal{R}$- and 
$\overline{\mathcal{R}}$-discriminants for several specific 
skew polynomial rings $\kk_{\bp}[x_1, x_2, x_3]$ over their 
centers. We ask the following question.

\begin{question}
\label{zzque4.7}
Let $A$ be a skew polynomial ring $\kk_\bp[x_1,\ldots,x_n]$
where each $p_{ij}$ is a root of unity. When does the 
$\mathcal{R}$-discriminant $\csr(A/Z)$ exist? In terms of 
$\bp$, what are the formulas for $\bar{\csr}(A/Z)$ and 
$\csr(A/Z)$?
\end{question}

A ring $R$ is called a \emph{noetherian unique factorization ring} 
(noetherian UFR) if $R$ is a prime left and right noetherian ring 
such that every nonzero prime ideal of $R$ contains a non-zero 
principal prime ideal. Hence, in a noetherian UFR every height one 
prime ideal is principal. Since the skew polynomial rings in 
Theorem~\ref{zzthm4.6} are noetherian hom-hom UFRs, part (2) of 
the following result is more general than Theorem~\ref{zzthm4.6}, 
but has a less constructive proof.

\begin{theorem}
\label{zzthm4.8}
Assume Hypothesis \ref{zzhyp1.1} for $(A,Z)$. Fix two positive 
integers $p$, $v$.
\begin{enumerate}
\item[(1)]
If $Z$ is a UFD, then the $\cR^p_v$-discriminant $\csr^{[p]}_v(A/Z)$ 
exists.
\item[(2)]
Suppose $A$ is a noetherian hom-hom UFR with affine center $Z$. 
Then the $\overline{\cR}^p_v$-discriminant 
$\bar{\csr}^{[p]}_v(A/Z)$ always exists.
\end{enumerate}
\end{theorem}

\begin{proof}
(1) This follows immediately since every reflexive ideal of 
a UFD is principal.

(2) Since $A$ is hom-hom, by Lemma~\ref{zzlem4.2}(3), $A$ is a 
tame $Z$-order and a reflexive $Z$-module. Thus, by \cite[Lemma 2.1]{SVdB}, every reflexive 
$Z$-module is reflexive as an $A$-module 
(see also \cite[Corollary 1.6]{silver}). By 
\cite[Lemma 3.3(ii)]{akalan}, every reflexive ideal of $A$ is 
principal. It follows that the $\overline{\cR}$-discriminant of 
$A$ over its center exists.
\end{proof}

Motivated by the above results, we ask the following questions.

\begin{question}
\label{zzque4.9}
Let $A$ be a hom-hom ring with center $Z$. Under what conditions 
does $\bar{\csr}(A/Z)$ (resp. $\csr(A/Z)$) exist? For example, if 
$A$ is Calabi--Yau, does $\bar{\csr}(A/Z)$ (resp. $\csr(A/Z)$) 
exist?
\end{question}

\section{Examples}
\label{zzsec5}

Much of the recent literature on discriminants (e.g., 
\cite{CPWZ1, CPWZ2, CYZ2, CYZ1}) has focused on connected graded 
hom-hom rings (namely, noetherian connected graded Artin--Schelter 
regular algebras satisfying a polynomial identity). In this paper 
we will test discriminant theory for a class of noncommutative 
algebras, the quantum generalized Weyl algebras (quantum GWAs), which 
are generally neither connected graded nor of finite global 
dimension. 

We begin by introducing GWAs in full generality and recalling some 
well-known facts regarding GWAs. We then proceed to work out the 
$\cR$-discriminant for quantum GWAs. 

\subsection{Properties of quantum GWAs}
\label{zzsec5.1}

The following definition is motivated by Bavula \cite{B1}.

\begin{definition}
\label{zzdef5.1}
Let $R$ be a $\kk$-algebra, 
$\sigma=(\sigma_1,\hdots,\sigma_\GWAdeg)$ a set of $\GWAdeg$ 
commuting automorphisms of $R$, and $h=(h_1,\hdots,h_\GWAdeg)$ 
a set of $\GWAdeg$ central regular elements in $R$ such that 
$\sigma_i(h_j)=h_j$ for $i,j \in \{1,\hdots,\GWAdeg\}$, 
$i \neq j$. With this data, we define the \emph{GWA of degree 
$\GWAdeg$} as the $\kk$-algebra  generated over $R$ as an algebra 
by $\bx=(x_1,\hdots,x_\GWAdeg)$ and $\by=(y_1,\hdots,y_\GWAdeg)$ 
with relations 
\begin{align*}
&x_ir = \sigma_i(r)x_i  &y_ir = \sigma_i\inv(r)y_i  
&	&\text{ for all $i \in \{1,\hdots, \GWAdeg\}$, $r \in R$} \\
&x_iy_i = h_i	&y_ix_i = \sigma_i\inv(h_i)	
&	&\text{ for all $i \in \{1,\hdots, \GWAdeg\}$} \\
&[x_i,x_j]=[y_i,y_j]=[x_i,y_j]=0	& 
&	&\text{ for all distinct $i,j \in \{1,\hdots, \GWAdeg\}$}.
\end{align*}
We denote this algebra by $R(\bx,\by,\sigma,h)$. We say 
$R(\bx,\by,\sigma,h)$ is a \emph{quantum GWA} if $R=\kk[t]$ and 
there exist nonzero scalars 
$q = (q_1,\hdots,q_\GWAdeg) \in (\kk\backslash\{0,1\})^\GWAdeg$ 
such that $\sigma_i(t) = q_i t$ for all $i \in \{1,\hdots,\GWAdeg\}$. 
If $q_i$ has finite order in the multiplicative group $\kk^\times$, 
then we let $\ord(q_i)$ denote the order of  $q_i$.
\end{definition}

In general, the base ring $R$ in  Definition~\ref{zzdef5.1} need 
not be commutative but in this paper, $R$ is always taken to be a 
commutative polynomial ring. We also only consider the case where 
each $q_i$ is a root of unity. More specifically, we will only 
consider the case in which the $\ord(q_i)$ are pairwise 
relatively prime.

If $W$ is a GWA of degree $\GWAdeg$, then for each $i = 1, \dots, m$ the subalgebra 
$W_i$ generated by $x_i$ and $y_i$ is a degree one GWA isomorphic to $R(x_i,y_i,\sigma_i,h_i)$ [Definition \ref{zzdef0.5}]. 
More generally, let $I \subseteq \{1,\hdots,\GWAdeg\}$. 
Set $\bx_I = (x_i)_{i \in I}$, $\by_I = (y_i)_{i \in I}$, 
$\sigma_I = (\sigma_i)_{i \in I}$, and $h_I = (h_i)_{i \in I}$. 
Then $\{\bx_I,\by_I\}$ generate a degree $|I|$ GWA over 
$R$ that we denote by $R(\bx_I, \by_I, \sigma_I,h_I)$. 

The class of GWAs is also closed under tensor products over $\kk$. 
If $W=R(\bx,\by,\sigma,h)$ and $W'=R'(\bx',\by',\sigma',h')$ are 
GWAs of degree $k$ and $\ell$, respectively, then 
\[W \tensor W' = (R \tensor R')((\bx,\bx'),(\by,\by'), 
(\sigma, \sigma'), (h, h') )\]
is a GWA of degree $k + \ell$. The following facts are well-known.

\begin{lemma}
\label{zzlem5.2}
Let $R$ be a $\kk$-algebra and suppose $W=R(\bx,\by, \sigma,h)$ 
is a degree $\GWAdeg$ GWA over $R$.
\begin{enumerate}
\item[(1)] 
If $R$ is a left {\rm{(}}right{\rm{)}} noetherian ring, then so 
is $W$.
\item[(2)] 
If $R$ is a domain, then so is $W$.
\item[(3)] 
There is a $\ZZ^{\GWAdeg}$-grading on $W$ obtained by setting 
$\deg(r)=\bzero$ for all $r \in R$ and $\deg(x_i)=\be_i$, 
$\deg(y_i)=-\be_i$ for all $i$.
\end{enumerate}
\end{lemma}

Suppose $W=R(\bx, \by, \sigma,h)$ is a GWA of degree $\GWAdeg$. 
For $\gamma =(\gamma_1,\ldots, \gamma_\GWAdeg)\in \ZZ^\GWAdeg$, 
let $z^\gamma$ denote $z_1^{\gamma_1}\ldots z_\GWAdeg^{\gamma_\GWAdeg}$ 
formally where $z_i^{\gamma_i} = x_i^{\gamma_i}$ when $\gamma_i \geq 0$ 
and $z_i^{\gamma_i} = y_i^{-\gamma_i}$ when  $\gamma_i < 0$. By 
a result of Benkart and Ondrus \cite[Proposition 2.5]{BO}, the 
center $Z(W)$ of $W$ is generated by 
$R^{\grp{\sigma}} = \{r \in R \mid \sigma_i(r) 
= r \text{ for all } i = 1,\hdots,\GWAdeg\}$ and those monomials 
$z^\gamma$ such that $\sigma^\gamma := 
\prod_{i=0}^{\GWAdeg} \sigma_i^{\gamma_i} = \Id_R$.
Note that this generalizes a result of Kulkarni 
\cite[Corollary 2.0.2]{kulk} in the case of a degree one GWA over 
a commutative domain.

In subsection \ref{zzsec5.3} we will use the following.

\begin{corollary}
\label{zzcor5.3}
Let $R$ a commutative domain and let $W=R(\bx, \by, \sigma,h)$ 
be a GWA of degree $\GWAdeg$ with $\GWAord_i = |\sigma_i|<\infty$ 
for all $i$. Suppose $\gcd(\GWAord_i,\GWAord_j)=1$ for all 
$i \neq j$. Then $Z(W)$ is generated over $R^{\grp{\sigma}}$ 
by $x_1^{\GWAord_1}, y_1^{\GWAord_1}, \hdots, 
x_\GWAdeg^{\GWAord_\GWAdeg}, y_\GWAdeg^{\GWAord_\GWAdeg}$.
\end{corollary}

\begin{proof}
By the hypotheses on the $\GWAord_i$, the subgroup of $\Aut(R)$ 
generated by the $\sigma_i$ is a finite abelian group of order 
$\GWAord=\prod_{i=1}^\GWAdeg  n_i$. The result now follows 
directly from \cite[Proposition 2.5]{BO}.
\end{proof}

An automorphism $\sigma$ of a $\kk$-algebra $R$ is called \emph{locally 
algebraic} if every finite-dimensional subspace of $R$ is
contained in a finite-dimensional $\sigma$-stable subspace of $R$. 
It is clear that if $\sigma$ has finite order, then $\sigma$ is 
locally algebraic.

The following result uses \cite{ebrahim} but it may also be 
possible to apply \cite{ZMZ} to the same effect.

\begin{lemma}
\label{zzlem5.4}
Let $R$ be an affine $\kk$-algebra and $W:=R(\bx, \by,\sigma,h)$ 
a degree $\GWAdeg$ GWA over $R$. If $|\sigma_i|<\infty$ for all 
$i$, then $\GKdim W = \GKdim R + \GWAdeg$.
\end{lemma}

\begin{proof}
Let $W(1)=R(x_1,y_1,\sigma_1,h_1)$. Since $|\sigma_1|<\infty$,
then $\sigma_1$ is locally algebraic. Hence, 
\cite[Theorem 26]{ebrahim} gives $\GKdim W(1) = \GKdim R + 1$. 

Let $k \geq 1$, $I=\{1,\hdots,k\}$, and 
$W(k)=R(\bx_I,\by_I,\sigma_I,h_I)$. Assume that $\GKdim W(k) 
= \GKdim R + k$. Recall that 
$W(k+1) \iso W(k)(x_{k+1},y_{k+1},\sigma_{k+1},h_{k+1})$
where $\sigma_{k+1}$ extends to $W(k)$ by setting
$\sigma_{k+1}(x_i)=x_i$ and $\sigma_{k+1}(y_i)=y_i$ for 
all $i<k+1$. Consequently, $|\sigma_{k+1}|$ has finite 
order as an automorphism of $W(k)$ and so $\sigma_{k+1}$ is 
locally algebraic on $W(k)$. Applying \cite[Theorem 26]{ebrahim} 
again gives
\[ \GKdim W(k+1) = \GKdim W(k) + 1 = \GKdim R + (k+1).\]
Now the assertion follows by induction.
\end{proof}

\begin{conjecture}
\label{zzcon5.5}
The conclusion of the previous lemma holds whenever $R$ is a 
finitely generated module over $R^G$ where 
$G=\grp{\sigma_1,\hdots,\sigma_\GWAdeg}$.
\end{conjecture}

The following is an extended version of \cite[Proposition 3.2]{KK}. 
We refer the reader there for the definition of 
Auslander--Gorenstein.

\begin{lemma}
\label{zzlem5.6}
Let $W=\kk[t](\bx,\by, \sigma,h)$ be a GWA of degree $\GWAdeg$.
Then $W$ is Auslander--Gorenstein and GK--Macaulay.
\end{lemma}

\begin{proof}
Set $S_0=\kk[t]$. For $i=1,\hdots,\GWAdeg$, let 
$S_i=S_{i-1}[x_i;\sigma_i]$ where $\sigma_i$ is extended to 
$S_{i-1}$ by setting $\sigma_i(x_j)=x_j$ for $j<i$. Set 
$T_0=S_\GWAdeg$. For $i=1,\hdots,\GWAdeg$, let 
$T_i=T_{i-1}[y_i;\sigma_i\inv,\delta_i]$ where $\sigma_i\inv$ 
is extended to $T_{i-1}$ by setting $\sigma_i\inv(y_j)=y_j$ for 
$j<i$, $\delta_i(r)=0$ for all $r \in R$, 
$\delta_i(y_j) = \delta_i(x_j)=0$ for $i \neq j$, and 
$\delta_i(x_i)=\sigma_i\inv(h_i)-h_i$. Set $T=T_\GWAdeg$.

Let $k_i=\deg_t(h_i)$. Filter $T$ by setting $\deg(t)=2$ and 
$\deg(x_i)=\deg(y_i)=k_i$. Then $\mathrm{gr}(T)$ is a connected 
graded ring which itself is an iterated Ore extension over 
$\kk[t]$. For each $i=1,\hdots,\GWAdeg$, $\overline{\sigma}_i$ 
preserves the grading on $\mathrm{gr}(T)$ so $\mathrm{gr}(T)$ 
is GK--Macaulay by \cite[Lemma (ii)]{LevaSt}. Thus, $T$ is 
Auslander--Gorenstein and GK--Macaulay by \cite[Lemma 4.4]{SZ}.
 
Since $\sigma_i(h_j)=h_j$ for $i \neq j$, then the elements 
$z_i=x_iy_i-h_i$ are central in $T$. Moreover, each $W_i$ is a 
domain and so $(z_1,\hdots,z_\GWAdeg)$ is a regular sequence of 
central elements. Hence, $T/(z_1,\hdots,z_\GWAdeg) \iso W$ is 
Auslander--Gorenstein and GK--Macaulay by \cite[Lemma (iii)]{LevaSt}.
\end{proof}

The global dimension of degree one GWAs has been studied 
extensively. Less is known in the case of higher degree GWAs. 
For our purposes, the next result is sufficient.

\begin{lemma}[{\cite{bavgldim,Jkrull}}]
\label{zzlem5.7}
Let $W=\kk[t](x,y,\sigma,h)$ be a quantum GWA with $q \neq 1$ a 
root of unity. If $h$ has multiple roots, then $\gldim W = \infty$. 
Otherwise, $\gldim W= 2$.
\end{lemma}

\subsection{The \texorpdfstring{$\cR$-discriminant}
{reflexive hull discriminant} of a degree one quantum GWA}
\label{zzsec5.2}

In this subsection, let $W=\kk[t](x,y,\sigma,h)$ denote a 
degree one quantum GWA with $n=|\sigma|<\infty$. We study the center of $W$ and the trace map, and 
conclude this subsection by computing the reflexive hull discriminant of $W$ 
over its center. This subsection serves as a warmup to the 
next subsection where we study higher degree quantum GWAs, but omit some details.

By \cite[Corollary 2.0.2]{kulk}, the center of $W$ is generated by 
$a=x^n,b= y^n, c=t^n$. Set $p(c)=\prod_{j=0}^{n-1}h^{\sigma^j}(t)$ 
where $h^{\sigma^j}(t) = h(\sigma^j(t))= h(q^j t)$. Then the 
generators satisfy a single relation as follows
\begin{equation}
\label{E5.7.1}\tag{E5.7.1} 
Z := Z(W)=\kk[a,b,c]/\left(ab - p(c)\right).
\end{equation}

\begin{lemma}
\label{zzlem5.8}
The center $Z$ given in \eqref{E5.7.1} is an affine normal
CM domain. As a consequence, Hypotheses \ref{zzhyp1.10}(1) and 
\ref{zzhyp2.1} hold for $(W,Z)$. 
\end{lemma}

\begin{proof} By definition, $Z$ has hypersurface singularities and so it is Gorenstein (and consequently, CM). It is clear that
$Z$ has isolated singularities. Hence $Z$ is normal by Serre's 
criterion for normality. 

It is easy to see that the rank of $W$ over $Z$ is $n^2$. Since 
$q$ is a primitive $n$th root of unity, ${\rm{char}}\; \kk$ does 
not divide ${\rm{rk}}_{Z}(W_Z)$. The consequence follows from 
\cite[Theorem 10.1]{Rei} and Lemma \ref{zzlem4.3}.
\end{proof}

Note that the form of the trace map $\tr$ is not essential in 
our proof of the main result of this section 
[Theorem \ref{zzthm5.11}]. However, a proof of Theorem 
\ref{zzthm5.11} can be done totally algebraically, in which 
case the trace map would be used in an essential way. In this 
particular case, it is easy to work out the trace map. We can 
write $W$ as a $Z$-algebra as follows
\begin{equation}
\label{E5.8.1}\tag{E5.8.1}
W\; =\; \frac{Z\langle x,y,t\rangle}{
(xt-qtx,yt-q^{-1}ty, xy-h(t), x^n-a,y^n-b,t^n-c)}.
\end{equation} 
Using this presentation, observe that $W$ is generated as a 
$Z$-module by the following elements 
\begin{equation}
\label{E5.8.2}\tag{E5.8.2}
\{x^it^j, y^it^j\mid i,j=0,\ldots,n-1\}.
\end{equation} 
Let $\tau=(1,t,t^2,\ldots,t^{n-1}) \in W^{\oplus n}$ and
\[
\mathbf{v}=(\tau,x\tau,x^2\tau,\ldots,x^{n-1}\tau,
y\tau,y^2\tau,\ldots,y^{n-1}\tau)\in W^{\oplus (2n^2-n)},
\]
considered as a column vector (or a $(2n^2-n)\times 1$ matrix).
The components of $\mathbf{v}$ generate $W$ as a module over $Z$ 
\eqref{E5.8.2}, and the first $n^2$ components of $\mathbf{v}$ 
form a semi-basis of $W$ (with respect to $\bv$). Consequently, 
the rank of $W$ over $Z$ is $n^2$. The modified discriminant 
ideal $\MD(W/Z)=\MD_{n^2}(W/Z)$ of $W$ is generated by the 
$n^2\times n^2$ minors of the $(2n^2-n)\times(2n^2-n)$ matrix
\[
D=\tr\cdot\mathbf{v}\mathbf{v}^T.
\]
We show that most of the entries of $D$ are zero in the next 
lemma. 

\begin{lemma}
\label{zzlem5.9}
We have $\tr(x^it^j)\ne0$ and $\tr(y^it^j)\ne 0$ if and only 
if $i\equiv 0$ mod $n$ and $j\equiv 0$ mod $n$. 
\end{lemma}

\begin{proof}
The $Z$-algebra $W$ has a $\ZZ/n\ZZ$-graded algebra structure, 
obtained by assigning the $Z$-algebra generators $x,y,t$ the 
degrees $1,-1,0$ respectively. Indeed, the relations provided 
in \eqref{E5.8.1} are all homogeneous under this grading. 

Left multiplication by $x^it^j$ shifts the degree by a nonzero 
integer unless $i\equiv0$ mod $n$. Therefore, if $i\not\equiv 0$ 
mod $n$, then $\tr(x^it^j)=0$. Next assume $i=kn$ for some 
integer $k$. Left multiplication by $x^{kn}t^j$ for 
$j\not\equiv 0$ mod $n$ permutes the semi-basis 
$\{x^it^j\}_{i,j=0}^{n-1}$ (up to scalar). Therefore, if 
$i,j\not\equiv 0$ mod $n$, then $\tr(x^i t^j)=0$.
Finally, if $i\equiv j\equiv 0$ mod $n$, then $x^it^j \in Z$ 
so left multiplication by $x^it^j$ can be represented by a 
diagonal matrix, hence $\tr(x^it^j) = n^2x^it^j$.  
The above argument still works if we replace all occurrences 
of $x$ by $y$. So this completes the proof.
\end{proof}

It is very complicated to compute the $\cR$-discriminant 
by using the definition, since $\MD(W/Z)$ is complicated.
So we will compute the $\cR$-discriminant locally first.
In the following lemma, we  use the elements of $\mathbf{v}$ 
to index the rows and columns of $D$, so that the first row 
of $D$ is the $1$ row, the second row is the $t$ row, and so on. 

\begin{lemma}
\label{zzlem5.10}
The following hold for the $(2n^2-n)\times(2n^2-n)$ square 
matrix $D=\tr\cdot \mathbf{v}\mathbf{v}^T$. 
\begin{enumerate}
\item[(1)] 
Let $0<i<n$. The only nonzero entries in the $x^it^j$ row of 
$D$ are in the $x^{n-i}t^{n-j}$ and $y^i t^{j'}$ columns for 
some $0\leq j'\leq n-1$.
\item[(2)] 
The statement in (1) holds if we switch $x$ and $y$. 
\item[(3)] 
The only nonzero entry in row $t^j$ is in column $t^{n-j}$.
\end{enumerate}
\end{lemma}

\begin{proof}
Consider the product $\rho = (x^i t^j) (y^k t^l)$. Using 
Lemma \ref{zzlem5.9}, and the grading on $W$ given in the 
proof, we see that $\tr(\rho)\ne 0$ implies that 
$i\equiv k$ mod $n$. 
Similarly, $\tr((x^i t^j) (x^k t^l))\ne 0$ if and only 
if $i+k\equiv 0$ mod $n$ and $j+l\equiv 0$ mod $n$.
The statement for $\tr(t^j)$ follows similarly.
\end{proof}

Recall from Lemma \ref{zzlem5.8} that $Z$ is a normal CM domain.

\begin{theorem}
\label{zzthm5.11} 
Retain the above notation. Then 
$\csr(W/Z)=_{Z^{\times}}c^{n(n-1)}=_{Z^{\times}}t^{n^2(n-1)}$. 
As a consequence, $\bar{\csr}(W/Z)=_{W^{\times}}c^{n(n-1)}
=_{W^{\times}}t^{n^2(n-1)}$.
\end{theorem}

\begin{proof}
Let $U_a$ and $U_b$ denote the open subsets of $X$ where 
$a\ne 0$ and $b\ne 0$ respectively. Note that the complement 
of $U:=U_a\cup U_b$ has codimension 2. It is easy to check that 
$W_a:=\restr{W}{U_a}=W[a^{-1}]$ is free of rank $n^2$ and has 
basis $\{x^it^j\mid i,j=0,\ldots,n-1\}$. By 
\cite[Proposition 1.4(2)]{CYZ1} or by a computation using 
Lemma \ref{zzlem5.10}(1,3), the discriminant of $W_a$ is given 
by $(ac)^{n(n-1)}=_{Z(W_a)^{\times}} c^{n(n-1)}$, up to 
some unit in $Z(W_a)$. Similarly the discriminant of $W_b$ is 
given by $(bc)^{n(n-1)}=_{Z(W_b)^{\times}} c^{n(n-1)}$, again 
up to some unit in $Z(W_b)$. This defines a Cartier divisor on 
$U$, which extends to a Cartier divisor on all of 
$X$. Indeed, the data $\{(U_a,c^{n(n-1)}), (U_b,c^{n(n-1)}), 
(X,c^{n(n-1)})\}$ defines a Cartier divisor on $X$ 
which restricts to the above Cartier divisor on $U$. 
The assertion now follows by Lemma \ref{zzlem2.2}.

For the consequence, we must show that $W$ is a reflexive 
$Z$-module. By Lemma \ref{zzlem5.6}, $W$ is GK--Macaulay, 
or equivalently, $W$ is CM. In particular, $W$ satisfies Serre's S2 
property. Since $\mathrm{Spec}(Z)$ is a normal surface, it 
satisfies the hypotheses of \cite[Proposition 1.9]{Ha}, so 
we conclude that $W$ is reflexive.
\end{proof}

\subsection{The \texorpdfstring{$\cR$-discriminant}
{reflexive hull discriminant} of a degree 
\texorpdfstring{$m$}{m} quantum GWA}
\label{zzsec5.3}

Our goal in this subsection is to generalize results from 
the previous subsection to higher degree quantum GWAs. This
result will be used in our applications.

Throughout this subsection, let $W=\kk[t](\bx,\by,\sigma,h)$ 
be a quantum GWA of degree $\GWAdeg$. As in Corollary 
\ref{zzcor5.3}, set $\GWAord_i = |\sigma_i|<\infty$ for 
all $i$ and assume $\gcd(\GWAord_i,\GWAord_j)=1$ for all 
$i \neq j$. Set $\GWAord =\GWAord_1 \cdots \GWAord_\GWAdeg$. 
The center $Z:=Z(W)$ of $W$ is generated by $a_i = x_i^{\GWAord_i}$, 
$b_i = y_i^{\GWAord_i}$, and $c=t^{\GWAord}$. Let 
\[ p_i(c) = \prod_{j=0}^{\GWAord_i-1}h_i^{\sigma_i^j}(t) \]
where $h_i^{\sigma_i^j}(t) = h_i(\sigma_i^j(t))= h_i(q_i^j t)$.
Thus, by Corollary \ref{zzcor5.3},
\[ Z = \frac{\kk[a_1,\hdots,a_\GWAdeg,b_1,\hdots,
b_\GWAdeg,c]}{\left(a_ib_i - p_i(c) \mid 1 \leq i \leq m \right)}. \]
In Theorem \ref{zzthm5.14} 
we show that the $\cR$-discriminant of $W$ is $c^{\GWAord(\GWAord-1)} 
\in Z$.

\begin{lemma}
\label{zzlem5.12}
Let $Z$ be the algebra defined as above. Then $Z$ is an affine 
normal CM domain. As a consequence, Hypotheses \ref{zzhyp1.10}(1) 
and \ref{zzhyp2.1} hold for $(W,Z)$. 
\end{lemma}

\begin{proof}
By definition, $Z$ is a complete intersection. Hence it is 
Gorenstein and CM. By the Jacobian criterion, the singular locus 
of $X:=\Spec Z$ has codimension $\geq 2$. Indeed, the Jacobian 
matrix of the defining ideal is given by the block matrix 
$J=(\mathrm{diag}(b_1,\cdots,b_m) | \mathrm{diag}(a_1,\cdots,a_m) | v)$ 
where $v$ is $m\times 1$ and depends only on $c$. The singular 
locus of $X$ occurs when $\mathrm{rank}(J)<m$, and a necessary 
condition for this is $a_i=b_i=0$ for some $i$. This shows that 
$\mathrm{codim}_X(X_{\mathrm{sing}}) \ge 2$, so $X$ is normal by 
Serre's criterion for normality. 
\end{proof}

We can write $W$ as a $Z$-algebra as follows:
\begin{equation}
\label{E5.12.1}\tag{E5.12.1}
W = \frac{Z\langle x_1\hdots, 
x_\GWAdeg,y_1,\hdots, y_\GWAdeg,t\rangle}{(x_it-q_itx_i,y_it-
q_i^{-1}ty_i, x_iy_i-h_i(t), x_i^{\GWAord_i}-a_i,y_i^{\GWAord_i}-b_i,
t^{\GWAord}-c)}.
\end{equation}

\begin{lemma}
\label{zzlem5.13}
Retain the above notation. Then $W$ is reflexive over $Z$.
\end{lemma}

\begin{proof}
The proof is almost identical to the dimension $2$ case 
(c.f. Theorem \ref{zzthm5.11}). Since $Z$ is a CM with
$\mathrm{codim}_X(X_{\mathrm{sing}}) \ge 2$ (where $X:=\Spec Z$),
the hypotheses of \cite[Proposition 1.9]{Ha} are satisfied. Then 
$W$ is $Z$-reflexive since it is a CM [Lemma \ref{zzlem5.6}] 
(hence S2) $Z$-module. 
\end{proof}

Here is the main result of this subsection. 

\begin{theorem}
\label{zzthm5.14}
Let $W$ be a quantum GWA of degree $\GWAdeg>1$ as given in 
\eqref{E5.12.1}. Then we have
$\csr(W/Z)=_{Z^{\times}}c^{n(n-1)}=_{Z^{\times}}t^{n^2(n-1)}.$ 
As a consequence, $\bar{\csr}(W/Z)=_{W^{\times}}c^{n(n-1)}
=_{W^{\times}}t^{n^2(n-1)}$.
\end{theorem}

\begin{proof}
Let $U$ denote the open set of $X:=\Spec Z$ with 
$\prod_{i=1}^m a_i\ne 0$. Then $\restr{W}{U}$ is free over its 
center, with basis given by $\{\mathbf{x}^\alpha t^j:=
x_1^{\alpha_1}x_2^{\alpha_2}\cdots x_m^{\alpha_m}t^j\}$ where 
$\alpha_i$ ranges from $0$ to $n_i-1$ and $j$ ranges from $0$ 
to $n-1$ (here $n=n_1\cdots n_m$). This shows that 
$\rk(\restr{W}{U})=n^2$.

We can use \eqref{E5.12.1} to compute the discriminant of 
$\restr{W}{U}$. Given a basis element $b=\mathbf{x}^\alpha t^j$, 
there is exactly one other basis element $b'$ whose product with 
$b$ has nonzero trace. Moreover $b'$ is given as follows
\begin{align*}
    b'&=\left\{
    \begin{array}{ll}
         1 & \text{if }\alpha=0,j=0 \\
         t^{n-j} &\text{if } \alpha=0,j\ne 0\\
         \mathbf{x}^{\alpha'}&\text{if }\alpha\ne 0,j=0\\
         \mathbf{x}^{\alpha'}t^{n-j}& \text{otherwise}
    \end{array}
    \right.
\end{align*}
where
\begin{align*}
    \alpha_j' &= \left\{
    \begin{array}{ll}
         0 & \text{if }\alpha_j=0\\
         n_j-\alpha_j&\text{otherwise.}
    \end{array}
    \right.
\end{align*}
This shows that the discriminant $d_U$ of $\restr{W}{U}$ is given by 
\begin{align*}
d_U&= c^{n(n-1)} \prod_{j=1}^m a_j^{n(n-n/n_j)}. 
\end{align*}
That is, $d_U=_{(\restr{W}{U})^{\times}} c^{n(n-1)}$. This can be 
obtained from a direct calculation. 

Now consider another open subset $V$ of $X$ where we replace 
the condition $\prod_{i=1}^m a_i\ne 0$ with $\prod_{i=1}^m a'_i\ne 0$
where $a'_i$ is either $a_i$ or $b_i$. By the symmetry of
$(x_i,y_i)$, $V$ is another $U$ after we switch some $x_i$
with $y_i$. Then $d_V$ can be obtained by the same computation
(by replacing the $a_i$'s with $a'_i$'s). Hence 
$d_V=_{(\restr{W}{V})^{\times}} c^{n(n-1)}$. Thus 
the data $\{(V,d_V:=c^{n(n-1)})\}$ (where $V$ ranges over all open 
subsets with $\prod_{i=1}^m a'_i\ne 0$) define a Cartier divisor 
on $X\backslash C$, where $C$ represents the union of subvarieties 
$\{a_i=b_i=0\}$ for $i=1,\ldots,m$. It is clear that $C$ has 
codimension $\geq 2$ in $X$. Just as in the degree one case, since 
$Z$ is an affine CM normal domain, the data $\{(V,c^{n(n-1)}), 
(X,c^{n(n-1)}\}$ extends to a Cartier divisor on all of $X$, so we 
are done by Lemma \ref{zzlem2.2}. 

As in Theorem \ref{zzthm5.11}, combining Lemma \ref{zzlem1.16} and 
Lemma \ref{zzlem5.13} implies that $W$ is reflexive and the second 
consequence follows.
\end{proof}

In some sense, quantum GWAs provide an ideal setting in which to 
study $\cR$-discriminants. This is because the center of a GWA is 
not a polynomial ring, but is reasonably nice and easy to compute. 
Other possible algebras to consider are quantized Weyl algebras and 
quantum matrix algebras.

\begin{question}
\label{zzque5.15}
For which other families of algebras is the $\cR$-discriminant 
computable? Furthermore, for which algebras is the reflexive hull 
of the modified discriminant ideal a principal ideal?
\end{question}

\section{Applications}
\label{zzsec6}

The purpose of this section is to show how to use 
$\cR$-discriminants to study important questions in noncommutative 
algebra. In particular, we consider the Automorphism Problem, the 
Isomorphism Problem, and the Zariski Cancellation Problem for 
quantum GWAs. These results are known in the degree one case and 
so our primary interest is extending them to higher degree quantum GWAs and 
tensor products of quantum GWAs. To avoid some degenerate cases, 
for a quantum GWA $W=\kk[t](\bx,\by,\sigma,h)$ we assume for the 
rest of this section that
\begin{equation}
\label{E6.0.1}\tag{E6.0.1}
h_i\not\in \kk \quad {\text{for every $i$}}
\end{equation}
as in many other papers. Note that \eqref{E6.0.1} implies that 
$W^{\times}=\kk^{\times}$.

\subsection{Automorphisms}
\label{zzsec6.1}
Let $W= \kk[t](x,y,\sigma,h)$ be a degree one quantum GWA. We 
recall the description of $\Aut(W)$ given by Su\'{a}rez--Alvarez 
and Vivas \cite{SAV}. Write $h = \sum_{i=0}^d c_i t^i$, where 
the $c_i \in \kk$ and $d=\deg_t(h)$, and set 
$g = \gcd\{i-j \mid i < j, c_ic_j \neq 0\}$. Let $C_g$ denote 
the subgroup of $\kk^\times$ consisting of $g$th roots of $1$.
If $h$ is a monomial, we make the convention that $g=0$ and 
$C_g=\kk^{\times}$. For each $\gamma \in C_g$ and each 
$\mu \in \kk^\times$, there is an automorphism $\eta_{\gamma,\mu}$ 
of $A$ given by
\[ \eta_{\gamma,\mu}(x) = \mu x, \quad 
\eta_{\gamma,\mu}(y) = \mu\inv \gamma^d y, \quad
\eta_{\gamma,\mu}(t) = \gamma t.\]
Let $G$ denote the subgroup of $\Aut(W)$ consisting of the 
$\eta_{\gamma,\mu}$. When $q=-1$, there is an automorphism 
$\Omega$ of $W$ defined by 
\[ \Omega(x)=y, \quad \Omega(y)=x, \quad \Omega(t)=-t.\]
We recover the following result of Su\'{a}rez--Alvarez and Vivas 
using the $\cR$-discriminant.

\begin{proposition} \cite[Theorem B]{SAV}
\label{zzthm6.1}
Let $q\in \kk$ be a root of unity and $W$ be a degree one quantum 
GWA. If $q^2\neq 1$, then $\Aut(W) \cong G$. If $q=-1$, then 
$\Aut(W) \cong G \rtimes \ZZ/2\ZZ$ where $\ZZ/2\ZZ$ is generated 
by $\Omega$.
\end{proposition}

\begin{proof}
It follows from \eqref{E6.0.1} that $W^{\times}=\kk^{\times}$. 
Let $\phi \in \Aut(W)$. By Theorem \ref{zzthm1.11}(1), $\phi$ 
fixes the $\cR$-discriminant up to a unit. Hence, 
$\phi(t^{n^2(n-1)}) =_{\kk^\times} t^{n^2(n-1)}$ by Theorem 
\ref{zzthm5.11}. Since $W$ is a ${\mathbb Z}$-graded domain with 
$\deg x=1$, $\deg y=-1$ and $\deg t=0$, $\deg \phi(t) = 0$. Then 
$\restr{\phi}{W_0}$ is an automorphism of $W_0 = \kk[t]$ so 
$\phi(t) = \gamma t$ for some $\gamma \in \kk^\times$.

Observe that $\phi(x)\phi(y) = \phi(xy) = \phi(h)$ is a 
homogeneous element of degree 0. Hence, $\phi(x)$ and $\phi(y)$ 
are homogeneous elements with $\deg \phi(x) = - \deg \phi(y)$. 
Since $\phi(x), \phi(y) ,\phi(t)$ generate $W$ as an algebra, we 
must have $\deg \phi(x) = \pm 1$. So $\phi(x) = \alpha x$ or 
$\phi(x) = \alpha y$ for some $\alpha \in \kk[t]$, $\alpha \neq 0$. 
Similarly, $\phi(y)=\beta y$ or $\phi(y)=\beta x$ for some 
$\beta \in \kk[t]$, $\beta \neq 0$. Thus,
\begin{align}
\label{E6.1.1}\tag{E6.1.1}
\deg_t(xy)=\deg_t(h) = \deg_t(\phi(h)) = \deg_t(\phi(xy)) = 
\deg_t(\alpha)+\deg_t(\beta)+\deg_t(xy),
\end{align}
so $\deg_t(\alpha)=\deg_t(\beta)=0$. Thus, $\alpha,\beta \in 
\kk^\times$. If $\phi(x) = \alpha y$, then $0 = \phi(xt - qtx) 
= \alpha \gamma (1-q^2) yt$, so $q^2 = 1$.

{\bf Case 1:} Assume $\phi(x)=\alpha x$ and $\phi(y)=\beta y$ 
for some $\alpha,\beta \in \kk^\times$. Write 
$h = \sum_{i=0}^d c_i t^i$ with $c_i \in \kk$ and $c_d \neq 0$. 
Then
\[\sum_{i=0}^d \alpha \beta c_i t^i =  (\alpha\beta) h 
= (\alpha\beta) xy = \phi(xy) = \phi(h) 
= \sum_{i=0}^d c_i \gamma^i t^i.\]
This implies that $\alpha\beta=\gamma^i$ for all $i$ such that 
$c_i \neq 0$. In particular, $\alpha\beta = \gamma^d$, so 
$\beta = \alpha\inv \gamma^d$. Now if $i < d$ and $c_i \neq 0$, 
then we have $\gamma^d = \gamma^i$, so $\ord(\gamma) \mid d-i$.
Hence, with $g = \gcd\{ d-i \mid i < d, a_i \neq 0\}$ as above, 
we must have $\gamma^g = 1$. It follows that $\phi = 
\eta_{\gamma,g} \in G$.
 
{\bf Case 2:} Assume $\phi(x)=\alpha y$ and $\phi(y)=\beta x$ 
for some $\alpha,\beta \in \kk^\times$, so that $q=-1$ by
the argument after \eqref{E6.1.1}. Now $\phi \circ \Omega$ 
is of the type in Case 1, so $\phi \circ \Omega = 
\eta_{-\gamma,g}$, whence $\phi = \eta_{-\gamma,g} \circ \Omega$. 
\end{proof}

Proposition \ref{zzthm6.1} extends easily to the degree $\GWAdeg$ 
case, as long as the orders of the automorphisms $\sigma_i$ are 
pairwise coprime.

\begin{proposition}
\label{zzpro6.2}
Let $W=\kk[t](\bx,\by,\sigma,h)$ be a quantum GWA of degree 
$\GWAdeg$. Suppose that each $q_i$ is a root of unity with 
$1<\GWAord_i=|\sigma_i|<\infty$ and $\gcd(\GWAord_i,\GWAord_j)=1$ 
for all $i \neq j$. Let $\phi \in \Aut(W)$. Then for each 
$i$, $\phi$ restricts to an automorphism of the quantum GWA 
subalgebra $W_i = \kk[t](x_i,y_i,\sigma_i,h_i)$. 
\end{proposition}

\begin{proof}
By Theorem \ref{zzthm5.14} and similar to the proof of 
Proposition \ref{zzthm6.1}, $\phi(t)=\gamma t$ for some 
$\gamma \in \kk$. For $i \in \{1,\hdots,n\}$, we have 
$\phi(x_iy_i)=\phi(h_i)$ and $\deg(h_i)=\bzero$. Using 
the $\ZZ^n$-grading on $W$, we have that there exists 
$j \in \{1,\hdots,n\}$ such that
\[
\text{(1)}~~\phi(x_i)=\alpha_i x_j, \phi(y_i)=\beta_i y_j 
\quad\text{or}\quad
\text{(2)}~~\phi(x_i)=\alpha_i y_j, \phi(y_i)=\beta_i x_j.
\]

Suppose we are in case (1). A similar argument holds in 
case (2). Then
\[ 0 = \phi(x_it-q_itx_i) = \alpha_i\gamma(q_j-q_i)tx_j.\]
By our hypothesis on the $q_i$, we have $q_i=q_j$. Hence, 
$j=i$. This shows that $\phi$ restricted to $W_i$ is of the 
form $\eta_{\gamma,g}$ or, when $q=-1$, possibly 
$\eta_{\gamma,g} \circ \Omega$.
\end{proof}

Let $W_1,\hdots,W_k$ be a collection of degree one quantum GWAs 
with canonical generators $\{x_i,y_i,t_i\}$ and parameters 
$\{q_i,h_i\}$. Set $A = W_1 \tensor \cdots \tensor W_k$. Note 
that $A$ is a degree $k$ GWA with base ring 
$\kk[t_1] \tensor \cdots \tensor \kk[t_k]$.
The next proposition should be compared to \cite[Theorem C]{LY} 
in the context of quantized Weyl algebras.

\begin{proposition}
\label{zzpro6.3}
Let $A = W_1 \tensor \cdots \tensor W_k$ as above. Assume each 
$q_i$ is a root of unity with $q_i^2 \neq 1$. Also assume 
$\deg_{t_i} h_i \geq 2$ for all $i$. If $\phi \in \Aut(A)$, 
then the following hold:
\begin{enumerate}
\item[(1)] 
There exists $\tau \in {\mathbb S}_k$ such that 
$\phi(W_i)=W_{\tau(i)}$ for all $i=1,\hdots,k$.
\item[(2)] 
There exists scalars 
$\alpha_1,\beta_1,\hdots,\alpha_k,\beta_k \in \kk^\times$ 
and a sequence $\{\epsilon_1,\hdots,\epsilon_k\} \in \{\pm 1\}^k$ 
such that for each $i=1,\hdots,k$,
\begin{align*}
&\phi(x_i)=\alpha_i x_{\tau(i)}, 
 \quad \phi(y_i)=\beta_i y_{\tau(i)}, 
 \quad \text{if $\epsilon_i=1$,} \\
&\phi(x_i)=\alpha_i y_{\tau(i)}, 
 \quad \phi(y_i)=\beta_i x_{\tau(i)}, 
 \quad \text{if $\epsilon_i=-1$}.
\end{align*}
Moreover, there exists scalars $\gamma_1,\hdots,\gamma_k 
\in \kk^\times$ such that 
\[ h_i(\gamma_i t) = \begin{cases}
\alpha_i\beta_i h_{\tau(i)}(t) & \text{if $\epsilon_i=1$} \\
\alpha_i\beta_i h_{\tau(i)}(q\inv t) & \text{if $\epsilon_i=-1$}.
\end{cases}\]
\end{enumerate}
\end{proposition}

\begin{proof}
Let $\okk$ be the algebraic closure of $\kk$. It is clear that we 
only need to show the statements for $A\otimes_{\kk}\okk$. In 
other words, we may assume $\kk$ is algebraically closed without 
loss of the generality. Let $Z$ be the center of $A$.

By Lemma \ref{zzlem5.6}, each $W_i$ is Auslander--Gorenstein and 
GK--Macaulay. Hence, $A$ is as well. Similarly, each $W_i$ is 
module-finite over its center, so $A$ is CM over $Z$. Together 
with Lemma \ref{zzlem5.8} and Theorem \ref{zzthm5.11} we may now 
apply Theorem \ref{zzthm2.5}.

By induction and Theorem \ref{zzthm2.5}, the $\cR$-discriminant 
of $A/Z$ is $\csr(A/Z)=t_1^{N_1} \cdots t_k^{N_k}=:d$ for some 
positive integers $N_i$. Let $\phi \in \Aut(A)$. By 
Theorem~\ref{zzthm1.11}, $\phi(d) =_{W^\times} d$ and since 
$W^\times = \kk^\times$, therefore $\phi(d)$ has degree $\bzero$ 
in the ${\mathbb Z}^k$-graded domain $A$. Since $\ZZ^k$ can be 
given a total order, this implies that each $\phi(t_i)$ is 
homogeneous of some $\ZZ^k$-grading. We claim that 
$\phi(t_i)=\gamma_i t_{\tau(i)}$ for some $\gamma_i \in \kk^\times$ 
and some $\tau \in {\mathbb S}_k$. 

After reordering, we may assume $N_1 \geq N_2 \geq \cdots \geq N_k$. 
Suppose $x_i^p$ is a factor of $\phi(t_1)$ for some 
$i \in \{1,\hdots,k\}$ and some $p>0$. Hence, $x_i^{pN_1}$ is a 
factor of $\phi(d)$. Since $\phi(d)$ has degree $\bzero$ in the 
$\ZZ^k$-grading, then it follows that $y_i^{pN_1}$ is a factor 
of $\phi(d)$. Then $(x_i^py_i^p)^{N_1}$ is a factor of $\phi(d)$. 
As $A_\bzero=\kk[t_1,\ldots,t_k]$ is a domain and $d$ is 
homogeneous in the $t_i$-grading, then $h_i=t_i^q$ for some 
$q\geq 2$. This implies $d$ has $t_i$-degree at least $pqN_1 > N_1$ 
in $d$, a contradiction. It follows that $\phi(t_1)=\gamma_1 t_i$ 
for some $i$. In this case we also obtain that $N_i=N_1$. The claim 
now follows by induction.

By the claim just proved above, $\phi$ preserves the polynomial 
subring $A_\bzero:=\kk[t_1,\ldots,t_k]$. Now,
\begin{equation}
\label{E6.3.1}\tag{E6.3.1}
h(\gamma_i t_{\tau(i)}) = h(\phi(t_i)) 
= \phi(h(t_i)) = \phi(x_iy_i) = \phi(x_i)\phi(y_i).
\end{equation}
It follows from the $\ZZ^k$-grading that 
$\deg(\phi(x_i))=\pm\be_{\tau(i)}$ for some 
$\tau\in {\mathbb S}_k$ and $\deg(\phi(x_i))=-\deg(\phi(y_i))$.
Furthermore, an argument as in \eqref{E6.1.1} show that 
$\phi(W_i)=W_{\tau(i)}$. Now it is straightforward to
fill out all details. 
\end{proof}

\subsection{Isomorphisms}
\label{zzsec6.2}
First we reprove one case of \cite[Theorem A]{SAV}.

\begin{proposition} \cite[Theorem A]{SAV}
\label{zzpro6.4}
Let $W=\kk[t](x,y,\sigma,h)$ and $W'=\kk[T](X,Y,\sigma',H)$ be 
degree one quantum GWAs where $\sigma(t)=qt$ and $\sigma'(T)=q'T$. 
Assume $q,q' \neq 1$ are roots of unity. If $\Phi:W \to W'$ is an 
isomorphism, then $q'=q^{\pm 1}$ and there exists 
$\gamma,\mu \in \kk^\times$ such that
\[ h(\gamma t) = \begin{cases}
	\mu H(T) & \text{ if $q'=q$} \\
	\mu H(q\inv T) & \text{ if $q'=q\inv$}.
\end{cases}\]
Moreover, if $q'=q^{\pm 1}$ and there exist $\gamma,\mu$ 
satisfying the above condition, then $W \iso W'$.
\end{proposition}

\begin{proof}
It is easy to verify that an isomorphism $W \to W'$ exists under 
these conditions. So assume $\Phi:W \to W'$ is an isomorphism.
An argument as in Proposition \ref{zzthm6.1} shows that 
$\Phi(t)=\gamma T$ for some $\gamma \in \kk^\times$, and that 
either $\Phi(x)=\alpha X$ and $\Phi(y)=\beta Y$, or else 
$\Phi(x)=\alpha Y$ and $\Phi(y)=\beta X$ for some 
$\alpha,\beta \in \kk^\times$. In the first case, we have 
\[ 0 = \Phi(xt-qtx) = \alpha\gamma(XT-qTX) 
= \alpha\gamma(q'-q)TX,\]
so $q'=q$, and
\[  h(\gamma t) = \Phi(h(t)) = \Phi(xy) = 
\Phi(x)\Phi(y) = (\alpha\beta)XY = (\alpha\beta)h(T).\]
The second case is similar.
\end{proof}

Now we extend this to higher degree quantum GWAs.

\begin{proposition}
\label{zzpro6.5}
Let $W=\kk[t](\sigma,h)$ and $W'=\kk[T](\sigma',H)$ be 
quantum GWAs of degree $\GWAdeg$ and $\GWAdeg'$, respectively, 
with parameters $(q_1,\hdots,q_{\GWAdeg})$ and 
$(q_1',\hdots,q_{\GWAdeg'}')$, respectively, such that 
$q_i^2, (q_i')^2 \neq 1$ for all $i$. Suppose both $W$ and $W'$ 
satisfy the hypothesis of Proposition \ref{zzpro6.2}. If 
$\phi:W \to W'$ is an isomorphism, then $\GWAdeg=\GWAdeg'$ and 
there exists $\tau \in {\mathbb S}_\GWAdeg$ such that 
$W_i \iso W_{\tau(i)}'$.
\end{proposition}

\begin{proof}
By Lemma \ref{zzlem5.4}, 
$\GWAdeg+1 = \GKdim W = \GKdim W' = \GWAdeg'+1$, so 
$\GWAdeg=\GWAdeg'$. Denote the canonical generators of $W$ 
by $x_1,y_1,\hdots,x_\GWAdeg,y_\GWAdeg$ and those of $W'$ 
by $X_1,Y_1,\hdots,X_\GWAdeg,Y_\GWAdeg$. Let $\phi:W \to W'$ 
be the given isomorphism, then using similar arguments to the 
above we have $\phi(t)=\gamma T$ for some $\gamma \in \kk^\times$. 
Then $\phi(x_i)\phi(y_i)$ has degree $\bzero$ for each $i$. Hence, 
there is some index $j$ such that $\phi(x_i) = \alpha_i X_j$ and 
$\phi(y_i) = \beta_i Y_j$, or else 
$\phi(x_i) = \alpha_i Y_j$ and $\phi(y_i) = \beta_i X_j$. 
We now refer to Proposition \ref{zzpro6.4} for a 
description of the isomorphism $W_i \to W_j$.
\end{proof}

\subsection{Cancellation}
\label{zzsec6.3}
The Zariski cancellation problem (ZCP) asks whether an algebra 
isomorphism $A[x] \iso B[x]$ implies the existence of an algebra 
isomorphism $A \iso B$; if so, then $A$ is called \emph{cancellative}
[Definition \ref{zzdef0.7}]. Solving the ZCP for various classes of 
noncommutative algebras has attracted much recent interest 
\cite{BHHV,BZ2,GXW,LWZ1,LWZ2}. In this subsection we use 
$\cR$-discriminants to prove that degree $m$ quantum GWAs, as well 
as their tensor products, are cancellative. For simplicity, we 
assume that ${\text{char}}\; \kk=0$ in this subsection.

Makar-Limanov showed that the cancellation property is inherently 
tied to the study of locally nilpotent derivations \cite{ML}. We we 
denote the set of locally nilpotent derivations of an algebra $A$ 
by $\LND(A)$. We say $A$ is \emph{LND-rigid} if $\LND(A)=\{0\}$ 
\cite[p.1711]{BZ2}.

A degree one quantum GWA $W$ is cancellative by 
\cite[Corollary 3.7 (2)]{BZ2}. Alternatively, one recovers this 
result by combining \cite[Theorem 3.6]{BZ2} and \cite[Lemma 2.1]{SAV}.

\begin{theorem}
\label{zzthm6.6}
Let $W=\kk[t](\bx,\by,\sigma,h)$ be a quantum GWA of degree 
$\GWAdeg$ and $Z$ its center. 
\begin{enumerate}
\item[(1)]
Suppose that each $q_i$ is a root of unity with 
$\GWAord_i=|\sigma_i|<\infty$, $\gcd(\GWAord_i,\GWAord_j)=1$ for 
all $i \neq j$, and $n_i>1$ for at least one $i$. Then $W$ is 
cancellative. 
\item[(2)]
Let $A$ be the tensor product of finitely many algebras in
part {\rm{(1)}}. Then $A$ is cancellative.
\end{enumerate}
\end{theorem}

\begin{proof}
(1) By \cite[Theorem 3.6]{BZ2}, it suffices to prove that $W$ 
is LND-rigid. Let $\delta \in \LND(W)$ and let $d$ denote the 
$\cR$-discriminant of $W$ given in Theorem \ref{zzthm5.14}. 
Since we assume \eqref{E6.0.1}, $W^{\times}=\kk^{\times}$.
By Theorem \ref{zzthm1.11}(2),  $\delta(d)=0$. By the 
Leibniz rule and because $W$ is a domain, $\delta(t)=0$.
Since $x_iy_i \in \kk[t]$, then by \cite[Lemma 7.4]{CYZ1} and 
\cite[p. 4]{ML}, $\delta(x_i)=\delta(y_i)=0$. Thus, 
$\LND(A)=\{0\}$. 

(2) The result for tensor products of quantum GWAs is similar.
First we need to pass to the case when $\kk$ is algebraically 
closed. Then $d:=\csr(A/Z(A))$ exists by Theorem \ref{zzthm2.5}(3)
and \ref{zzthm5.14}. Then we can copy the proof of part (1)
with minor changes.
\end{proof}

\bibliographystyle{plain}

\end{document}